\documentclass[12pt,oneside]{amsart}

\title[Rigidity results for free boundary hypersurfaces]{Rigidity results for free boundary hypersurfaces in initial data sets with boundary}

\author{Deivid de Almeida}
\address{Instituto Federal de Educação, Ciência e Tecnologia do Ceará, Crateús, CE, Brazil.}
\email{deivid.almeida@ifce.edu.br}

\author{Abraão Mendes}
\address{Instituto de Matemática, Universidade Federal de Alagoas, Maceió, AL, Brazil.}
\email{abraao.mendes@im.ufal.br}

\usepackage[margin=2.5cm]{geometry}

\usepackage{lmodern}

\usepackage[none]{hyphenat}

\usepackage{setspace}

\newtheorem{theo}{Theorem}[section]
\newtheorem{theointro}{Theorem}
\newtheorem{theointroletter}{Theorem}

\newtheorem{lem}[theo]{Lemma}
\newtheorem{prop}[theo]{Proposition}
\newtheorem{cor}[theo]{Corollary}
\theoremstyle{remark}
\newtheorem{rmk}[theo]{Remark}
\newtheorem{example}[theo]{Example}

\newcommand{\g}{\mathbf{g}}

\usepackage{color}

\DeclareMathOperator{\divergente}{div}
\DeclareMathOperator{\II}{II}

\renewcommand{\div}{\divergente}

\numberwithin{equation}{section}

\newcommand{\ds}{\displaystyle}

\usepackage{hyperref}

\begin{document}

\raggedbottom

\setstretch{1.2}

\begin{abstract}
In this work, we present several rigidity results for compact free boundary hypersurfaces in initial data sets with boundary. Specifically, in the first part of the paper, we extend the local splitting theorems from [G.~J.~Galloway and H.~C.~Jang, \textit{Some scalar curvature warped product splitting theorems}, Proc.\ Am.\ Math.\ Soc.\ \textbf{148} (2020), no.~6,\linebreak 2617--2629] to the setting of manifolds with boundary. To achieve this, we build on the approach of the original paper, utilizing results on free boundary marginally outer trapped surfaces (MOTS) applied to specific initial data sets. In the second part, we extend the main results from [A.~Barros and C.~Cruz, \textit{Free boundary hypersurfaces with non-positive Yamabe invariant in mean convex manifolds}, J.\ Geom.\ Anal.\ \textbf{30} (2020), no.~4, 3542--3562] to the context of free boundary MOTS in initial data sets with boundary.
\end{abstract}

\maketitle

\section{Introduction}

In differential geometry, there are several results aimed at understanding the geometry and topology of Riemannian manifolds that have a lower bound on their scalar curvature. In the case of dimension \(n = 2\), a classical result by R.~Schoen and S.-T.~Yau~\cite{ShoenYau1979} shows that any closed surface that minimizes area in a three-dimensional Riemannian manifold with positive scalar curvature is homeomorphic to either \(\mathbb{S}^2\) or \(\mathbb{RP}^2\).

In this context, another result due to Schoen and Yau~\cite{ShoenYau1987} is the following:

\begin{theointro}[Schoen-Yau, 1987]
Let \((M^{n+1}, g)\), with \(n \geq 2\), be a Riemannian manifold with positive scalar curvature, \(R^M > 0\). If \(\Sigma\) is a closed, two-sided, stable minimal hypersurface in \(M\), then \(\Sigma\) admits a metric of positive scalar curvature.
\end{theointro}

In~\cite{Cai2002}, M.~Cai proved the following splitting theorem, assuming that \(\Sigma\) minimizes volume rather than merely being stable (see \cite{Galloway2011} for a simplified proof).

\begin{theointro}[Cai, 2002]\label{theo:Cai}
Let \((M^{n+1}, g)\), with \(n \geq 2\), be a Riemannian manifold with non-negative scalar curvature, \(R^M \geq 0\). Suppose \(\Sigma\) is a closed, two-sided minimal\linebreak hypersurface that minimizes volume in \(M\). If \(\Sigma\) does not admit a metric of positive scalar curvature, then there exists a neighborhood \(V\) of \(\Sigma\) in $M$ such that \((V, g|_V)\) is isometric to \((-\delta,\delta)\times\Sigma\) with the product metric \(dt^2 + h\), where \(h = g|_\Sigma\) and \((\Sigma, h)\) is Ricci-flat.
\end{theointro}

The case of the theorem above for \(n = 2\) was proved by M.~Cai and G.~J.~Galloway~\cite{CaiGalloway2000}\linebreak (see the comment following Theorem~\ref{theo:Nunes} below).

More recently, G.~J.~Galloway and H.~C.~Jang~\cite{GallowayJang2020} obtained a splitting theorem for a neighborhood of the boundary \(\Sigma\) of a Riemannian manifold \((M, g)\), assuming a lower bound on the scalar curvature of \(M\) and an upper bound on the mean curvature of \(\Sigma\), while also assuming that \(\Sigma\) does not admit a metric of positive scalar curvature. To prove this theorem, they used several results from the theory of marginally outer trapped surfaces (MOTS), considering the special initial data set \((M, g, K = -\epsilon g)\), where \(\epsilon = 0\) or \(1\). More precisely, they proved the following result:

\begin{theointro}[Galloway-Jang, 2020]\label{theo:Galloway-Jang}
Let \((M^{n+1}, g)\), with \(n \geq 2\), be a Riemannian manifold with compact boundary \(\Sigma\). Fix \(\epsilon = 0\) or \(1\), and assume the following conditions:
\begin{enumerate}
\item[$(1)$] The scalar curvature of \(M\) satisfies \(R^M \geq -(n+1)n\epsilon\);
\item[$(2)$] The mean curvature of \(\Sigma\) satisfies \(H^\Sigma \leq n\epsilon\);
\item[$(3)$] \(\Sigma\) does not admit a metric of positive scalar curvature;
\item[$(4)$] \(\Sigma\) is locally weakly outermost.\footnote{In Theorem~\ref{theo:Galloway-Jang}, the term ``locally weakly outermost'' corresponds to being ``locally weakly outermost with respect to $H_0 = n\epsilon$'' in our terminology (see \cite[p.~2]{GallowayJang2020}).}
\end{enumerate}
Under these assumptions, there exists a neighborhood \(V\) of \(\Sigma\) in \(M\) such that \((V, g|_V)\) is isometric to \([0, \delta) \times \Sigma\) with the warped product metric \(dt^2 + e^{2\epsilon t} h\), where \(h = g|_\Sigma\) and \((\Sigma, h)\) is Ricci-flat.
\end{theointro}

They also observed that an analogous result holds if the condition on the mean curvature of \(\Sigma\) is replaced by \(H^\Sigma \leq -n\epsilon\). In this case, \((V, g|_V)\) will be isometric to \(([0, \delta) \times \Sigma, dt^2 + e^{-2\epsilon t}h)\).
 
Below, we present the first result of this work, inspired by Theorem~\ref{theo:Galloway-Jang}, the proof of which will be provided in Section~\ref{sec:proof.Theo.A}.

\begin{theointroletter}\label{theo:A}
Let \((M^{n+1}, g)\), with \(n \geq 2\), be a Riemannian manifold with boundary, and let \(\Sigma\) be a compact, free boundary hypersurface in \(M\). Fix \(\epsilon = 0\) or \(1\), and assume the following conditions:
\begin{enumerate}
\item[$(1)$] The manifold \(M\) has scalar curvature \(R^M \geq -(n+1)n\epsilon\), and its boundary has mean curvature \(H^{\partial M} \geq 0\);
\item[$(2)$] The hypersurface \(\Sigma\) has mean curvature \(H^{\Sigma} \leq n\epsilon\);
\item[$(3)$] \(\Sigma\) does not admit a metric with positive scalar curvature and minimal boundary;
\item[$(4)$] \(\Sigma\) is locally weakly outermost with respect to \(H_0 = n\epsilon\).
\end{enumerate}
Then, there exists an outer neighborhood \(V\) of \(\Sigma\) in \(M\) such that \((V, g|_V)\) is isometric to \([0, \delta) \times \Sigma\) with the warped product metric \(dt^2 + e^{2\epsilon t} h\), where \(h = g|_\Sigma\) and \((\Sigma, h)\) is Ricci-flat with a totally geodesic boundary.
\end{theointroletter}

As in the closed case, an analogous result holds if we replace \( H^\Sigma \leq n\epsilon \) and \( dt^2 + e^{2\epsilon t} h \) above by \( H^\Sigma \leq -n\epsilon \) and \( dt^2 + e^{-2\epsilon t} h \), respectively, assuming that $\Sigma$ is locally weakly outermost with respect to $H_0=-n\epsilon$.

At the end of Section~\ref{sec:proof.Theo.A}, we present examples that illustrate the necessity of condition~(3) (for $\epsilon = 0$ or $1$) and condition~(4) (for $\epsilon = 1$) in Theorem~\ref{theo:A}. These examples are inspired by those of Galloway and Jang (see~\cite[Remark~1]{GallowayJang2020}) in the closed case.

As noted, rigidity results involving scalar curvature have garnered significant attention from geometers over the years. The following theorem, due to H.~Bray, S.~Brendle, and A.~Neves~\cite{BrayBrendleNeves2010}, concerns surfaces and assumes the existence of an embedded sphere that locally minimizes area in a Riemannian manifold with scalar curvature bounded below by a positive constant. We restate it as follows:

\begin{theointro}[Bray-Brendle-Neves, 2010] 
Let \((M^3, g)\) be a Riemannian manifold with scalar curvature \(R^M \ge 2c\), for some constant \(c > 0\). If \(\Sigma^2 \subset M^3\) is an embedded sphere that locally minimizes area, then the area of \(\Sigma\) satisfies
$$
A(\Sigma) \leq \frac{4 \pi}{c}.
$$
Moreover, if equality holds, then there exists a neighborhood \(V\) of \(\Sigma\) in \(M\) such that \((V, g|_V)\) is isometric to \((- \delta, \delta) \times \Sigma\) with the product metric \(dt^2 + h\), where \(h = g|_\Sigma\) and \((\Sigma, h)\) is the round sphere of radius \(1/\sqrt{c}\). In particular, if \(M\) is complete and \(\Sigma\) minimizes area in its isotopy class, then the universal cover of \(M\) is isometric to the product \((\mathbb{R} \times \Sigma, dt^2 + h)\), assuming equality holds.
\end{theointro}

The next theorem, due to I.~Nunes~\cite{Nunes2013}, extends the result of Bray, Brendle, and Neves to the context of closed surfaces \(\Sigma\), with genus \(\g(\Sigma) \geq 2\), embedded in a Riemannian manifold with scalar curvature bounded below by a negative constant.

\begin{theointro}[Nunes, 2013]\label{theo:Nunes}
Let \((M^3, g)\) be a Riemannian manifold with scalar curvature \(R^M \ge -2c\), for some constant \(c > 0\). If \(\Sigma^2 \subset M^3\) is a two-sided, embedded, closed Riemann surface with genus \(\g(\Sigma) \geq 2\) that locally minimizes area, then the area of \(\Sigma\) satisfies
$$
A(\Sigma) \geq \frac{4 \pi (\g(\Sigma) - 1)}{c}.
$$
Moreover, if equality holds, then there exists a neighborhood \(V\) of \(\Sigma\) in \(M\) such that \((V, g|_V)\) is isometric to \((- \delta, \delta) \times \Sigma\) with the product metric \(dt^2 + h\), where \(h = g|_\Sigma\) and \((\Sigma, h)\) has constant Gaussian curvature equal to \(-c\). In particular, if \(M\) is complete and \(\Sigma\) minimizes area in its isotopy class, then the universal cover of \(M\) is isometric to the product \((\mathbb{R} \times \Sigma, dt^2 + h)\), assuming equality holds.
\end{theointro}

It is worth noting that the works of Cai~\cite{Cai2002}, Bray-Brendle-Neves~\cite{BrayBrendleNeves2010}, and Nunes~\cite{Nunes2013} were, in part, inspired by the pioneering work of Cai and Galloway~\cite{CaiGalloway2000}, in which, motivated by questions related to the topology of black holes, they solved a problem posed by\linebreak  D.~Fischer-Colbrie and R.~Schoen~\cite{Fischer-ColbrieSchoen1980}. 
Among other things, they proved the rigidity result that we paraphrase as follows: if \((M^3, g)\) is a Riemannian manifold with non-negative scalar curvature and \(\Sigma^2 \subset M^3\) is a two-sided embedded torus that locally minimizes area, then \((M,g)\) is isometric to \((- \delta, \delta) \times \Sigma\) with the product metric \(dt^2 + h\) in a neighborhood of \(\Sigma\), where \(h = g|_\Sigma\) and \((\Sigma, h)\) is a flat torus (\emph{i.e.}\ \(M\) is flat in a neighborhood of \(\Sigma\)). Moreover, if \(M\) is complete and \(\Sigma\) minimizes area in its isotopy class, then the universal cover of \(M\) is isometric to \((\mathbb{R} \times \Sigma, dt^2 + h)\) (\emph{i.e.}\ \(M\) is globally flat).

Despite the similarities between the results due to Bray-Brendle-Neves, Nunes, and \linebreak Cai-Galloway, 
their proofs are quite different. However, in 2015, M.~Micallef and V.~Moraru, in~\cite{MicallefMoraru2015}, presented a unified proof for these important results. This allowed, one year later, V.~Moraru~\cite{Moraru2016} to extend Nunes' theorem to the case of closed hypersurfaces \(\Sigma\) of dimension \(n \geq 3\).

\begin{theointro}[Moraru, 2016]\label{theo:Moraru}
Let $(M^{n+1}, g)$, with $n \geq 3$, be a Riemannian manifold with scalar curvature $R^M \geq -2c$, for some constant $c > 0$. Let $\Sigma^n \subset M^{n+1}$ be a two-sided, closed embedded hypersurface with $\sigma(\Sigma) < 0$ that locally minimizes volume. Then the volume of $\Sigma$ satisfies
$$
\operatorname{vol}(\Sigma) \geq \left(\frac{|\sigma(\Sigma)|}{2c}\right)^{\frac{n}{2}}.
$$
Moreover, if equality holds, then there exists a neighborhood $V$ of $\Sigma$ in $M$ such that $(V, g|_V)$ is isometric to $(-\delta, \delta) \times \Sigma$ with the product metric $dt^2 + h$, where $h = g|_\Sigma$ and $(\Sigma, h)$ is Einstein with scalar curvature $R^\Sigma = -2c$. 
\end{theointro}

Above, \(\sigma(\Sigma)\) represents the Yamabe invariant of \(\Sigma\) (see Subsection~\ref{subsec:preliminares2}).

The case of closed hypersurfaces \(\Sigma\) of dimension \(n \geq 2\) with \(\sigma(\Sigma) \leq 0\) (\emph{i.e.}\ those that do not admit a metric of positive scalar curvature), in Riemannian manifolds with non-negative scalar curvature, corresponds to Theorem~\ref{theo:Cai}. For hypersurfaces of dimension \(n \geq 3\) with \(\sigma(\Sigma) > 0\) in Riemannian manifolds with positive scalar curvature, only a few special cases have been addressed so far: A.~Barros \emph{et al.}~\cite{BarrosCruzBatistaSousa2015} studied the case of hypersurfaces \(\Sigma\) of dimension \(n = 4\), assuming that \((\Sigma, h)\) is Einstein; the second-named author~\cite{Mendes2019} extended the work of A.~Barros \emph{et al.}\ to more general (not necessarily Einstein) hypersurfaces of dimension \(n = 4\); and H.~Deng~\cite{Deng2021} studied the case when \(n > 4\) and \((\Sigma, h)\) is Einstein.

In 2020, A.~Barros and C.~Cruz~\cite{BarrosCruz2020} extended Theorem~\ref{theo:Moraru} to the setting of compact free boundary hypersurfaces in Riemannian manifolds with boundary. The result is as follows:

\begin{theointro}[Barros-Cruz, 2020]\label{theo:Barrosetal} 
Let $(M^{n+1}, g)$, with $n \geq 3$, be a Riemannian manifold with scalar curvature $R^M$ bounded below and mean convex boundary $\partial M$, \emph{i.e.}\ $H^{\partial M} \geq 0$. Consider a compact, two-sided, properly embedded hypersurface $\Sigma^n \subset M^{n+1}$ that is free boundary and locally minimizes volume.
\begin{enumerate}
\item[I)] If $\inf R^M < 0$ and $\sigma^{1,0}(\Sigma, \partial \Sigma) < 0$, then
$$
\operatorname{vol}(\Sigma) \geq \left(\frac{\sigma^{1,0}(\Sigma, \partial \Sigma)}{\inf R^M}\right)^{\frac{n}{2}}.
$$
Furthermore, if equality holds, then there exists a neighborhood $V$ of $\Sigma$ in $M$ such that $(V, g|_V)$ is isometric to $(-\delta, \delta) \times \Sigma$ with the product metric $dt^2 + h$, where $h = g|_\Sigma$ and $(\Sigma, h)$ is Einstein with scalar curvature $R^\Sigma = \inf R^M$ and a totally geodesic boundary.
\item[II)] If $R^M \geq 0$ and $\sigma^{1,0}(\Sigma, \partial \Sigma) \leq 0$, then there exists a neighborhood $V$ of $\Sigma$ in $M$ such that $(V, g|_V)$ is isometric to $(-\delta, \delta) \times \Sigma$ with the product metric $dt^2 + h$, where $h = g|_\Sigma$ and $(\Sigma, h)$ is Ricci-flat with a totally geodesic boundary.
\end{enumerate}
\end{theointro}

Above, $\sigma^{1,0}(\Sigma, \partial \Sigma)$ denotes the Yamabe invariant of the compact manifold $\Sigma$ with boundary (see Subsection~\ref{subsec:preliminares2}).

Recently, L.~F.~Pessoa, E.~Véras, and B.~Vieira~\cite{PessoaVerasVieira} extended the results of~\cite{BarrosCruz2020} to the setting of capillary constant mean curvature hypersurfaces.

In the second part of this work, we extend Theorem~\ref{theo:Barrosetal}, item I), to the context of free boundary marginally outer trapped surfaces (MOTS).

Over the years, the study of MOTS, motivated by the theory of relativity, has garnered significant attention from the mathematical academic community, particularly in the field of differential geometry. Within this context, and concerning rigidity results, we present the following theorem due to the second-named author~\cite{Mendes2019TAMS} (definitions will be provided in Subsection~\ref{subsec:preliminares1}):

\begin{theointro}[Mendes, 2019]\label{theo:Mendes1}
Let $(M^3,g,K)$ be an initial data set, and let $\Sigma^2$ be a closed, weakly outermost MOTS in $(M^3,g,K)$ with genus $\g(\Sigma) \geq 2$. If $\mu - |J| \geq -c$ for some constant $c > 0$ and $K$ is 2-convex on $M_+$, then the area of $\Sigma$ satisfies
$$
A(\Sigma) \geq \frac{4\pi(\g(\Sigma) - 1)}{c}.
$$
Moreover, if equality holds, then:
\begin{enumerate}
\item[$(1)$] There exists an outer neighborhood $V$ of $\Sigma$ in $M$ such that $(V, g|_V)$ is isometric to $[0, \delta) \times \Sigma$ with the product metric $dt^2 + h$, where $h = g|_\Sigma$ and $(\Sigma, h)$ has constant Gaussian curvature equal to $-c$;
\item[$(2)$] $K = a \, dt^2$ on $V$, where $a \in C^\infty(V)$ depends only on $t \in [0, \delta)$;
\item[$(3)$] $\mu = -c$ and $J = 0$ on $V$.
\end{enumerate}
\end{theointro}

For hypersurfaces of dimension \( n \geq 3 \), he established the following result:

\begin{theointro}[Mendes, 2019]\label{theo:Mendes2}
Let $(M^{n+1},g,K)$, with \( n \geq 3 \), be an initial data set, and let $\Sigma^n$ be a closed, weakly outermost MOTS in $(M^{n+1},g,K)$ with Yamabe invariant $\sigma(\Sigma) < 0$. If $\mu - |J| \geq -c$ for some constant \( c > 0 \) and $K$ is $n$-convex on $M_+$, then the volume of $\Sigma$ satisfies
$$
\operatorname{vol}(\Sigma) \geq \left(\frac{|\sigma(\Sigma)|}{2c}\right)^{\frac{n}{2}}.
$$
Moreover, if equality holds, then:
\begin{enumerate}
\item[$(1)$] There exists an outer neighborhood $V$ of $\Sigma$ in $M$ such that $(V, g|_V)$ is isometric to $[0, \delta) \times \Sigma$ with the product metric $dt^2 + h$, where $h = g|_\Sigma$ and $(\Sigma, h)$ is Einstein with scalar curvature $R^\Sigma = -2c$;
\item[$(2)$] $K = a \, dt^2$ on $V$, where $a \in C^\infty(V)$ depends only on \( t \in [0, \delta) \);
\item[$(3)$] $\mu = -c$ and $J = 0$ on $V$.
\end{enumerate}
\end{theointro}

Recently, A.~Alaee, M.~Lesourd, and S.-T.~Yau~\cite{AlaeeLesourdYau2021} extended the concept of free boundary minimal surfaces to marginally outer trapped surfaces (MOTS). Among other results, they proved an analog of Theorem~\ref{theo:Mendes1} for compact free boundary MOTS.

The second result in this work was inspired by the aforementioned results, extending part of the work of Alaee, Lesourd, and Yau to compact free boundary MOTS \(\Sigma^n\), with \( n \geq 3 \), and Yamabe invariant \(\sigma^{1,0}(\Sigma, \partial \Sigma) < 0\). The case of compact free boundary MOTS with Yamabe invariant \(\sigma^{1,0}(\Sigma, \partial \Sigma) \leq 0\) (\emph{i.e.}\ those that do not admit a positive scalar curvature metric with minimal boundary) was recently addressed by the second-named author~\cite{Mendes2022}, which corresponds to the analog of Theorem~\ref{theo:Barrosetal}, item II), for compact free boundary MOTS.

We now state our second result (definitions will be provided in Section~\ref{sec:preliminaries}):

\begin{theointroletter}\label{theo:B}
Let $(M^{n+1}, g, K)$, with $n \geq 3$, be an initial data set with boundary, and let $\Sigma^n$ be a compact, free boundary stable MOTS in $(M^{n+1}, g, K)$ that is weakly outermost and whose Yamabe invariant $\sigma^{1,0}(\Sigma, \partial \Sigma)$ is negative. 

If $\mu - |J| \geq -c$ for some constant $c > 0$, $(M, g, K)$ satisfies the boundary dominant energy condition, and $K$ is $n$-convex, all of which hold on $M_+$, then the volume of $\Sigma$ satisfies
\[
\operatorname{vol}(\Sigma) \geq \left( \frac{|\sigma^{1,0}(\Sigma, \partial \Sigma)|}{2c} \right)^{\frac{n}{2}}.
\]

Moreover, if equality holds, then:
\begin{enumerate}
 \item[$(1)$] There exists an outer neighborhood $V$ of $\Sigma$ in $M$ such that $(V, g|_V)$ is isometric to $[0, \delta) \times \Sigma$ with the product metric $dt^2 + h$, where $h = g|_\Sigma$ and $(\Sigma, h)$ is Einstein with scalar curvature $R^\Sigma = -2c$ and a totally geodesic boundary;
 \item[$(2)$] $K = a \, dt^2$ on $V$, where $a \in C^\infty(V)$ depends only on $t \in [0, \delta)$;
 \item[$(3)$] $\mu = -c$ and $J = 0$ on $V$;
 \item[$(4)$] The boundary dominant energy condition is saturated along $V \cap \partial M$.
\end{enumerate}
\end{theointroletter}

At the end of Section~\ref{sec:proof.Theo.B}, we present an initial data set model in support of Theorem~\ref{theo:B}.

The paper is organized as follows: In Subsection~\ref{subsec:preliminares1}, we present some preliminaries for Theorem~\ref{theo:A} and, in Subsection~\ref{subsec:preliminares2}, for Theorem~\ref{theo:B}. In Section~\ref{sec:proof.Theo.A}, we present the proof of Theorem~\ref{theo:A} and a corollary of it. Finally, in Section~\ref{sec:proof.Theo.B}, we present the proof of Theorem~\ref{theo:B}.

\medskip

\textbf{Acknowledgments.} The authors sincerely thank Greg Galloway for his kind interest in this work. This research forms part of the first author's Ph.D. thesis, supported in part by the Coordenação de Aperfeiçoamento de Pessoal de Nível Superior (CAPES), Brazil, and the Instituto Federal do Ceará (IFCE), Campus Crateús, Brazil. The second author acknowledges partial support from the Conselho Nacional de Desenvolvimento Científico e Tecnológico (CNPq), Brazil (Grants 309867/2023-1 and 405468/2021-0), as well as from the Fundação de Amparo à Pesquisa do Estado de Alagoas (FAPEAL), Brazil (Process E:60030.0000002254/2022).

\section{Preliminaries}\label{sec:preliminaries}

\subsection{Free boundary MOTS}\label{subsec:preliminares1}

An \emph{initial data set} $(M, g, K)$ consists of a Riemannian\linebreak manifold $(M, g)$ and a symmetric $(0, 2)$-tensor $K$ defined on $M$. For $(M, g, K)$, the \emph{local energy density} and the \emph{local current density} are defined as
\[
\mu = \frac{1}{2}( R^M - |K|^2 + (\operatorname{tr}K)^2) 
\quad \text{and} \quad 
J = \operatorname{div}(K - (\operatorname{tr}K)g),
\]
respectively, where $R^M$ is the scalar curvature of $(M, g)$. The initial data set $(M, g, K)$ is said to satisfy the \emph{dominant energy condition} (DEC) if
\[
\mu - |J| \geq 0 \quad \text{on} \quad M.
\]

In this work, we assume that $M$ is a differentiable manifold with boundary. The second fundamental form of $\partial M$ in $(M, g)$, denoted by $\mathrm{II}^{\partial M}$, is defined as
\[
\mathrm{II}^{\partial M}(Y, Z) = \langle \nabla_Y \varrho, Z \rangle, \quad Y, Z \in \mathfrak{X}(\partial M),
\]
where $\varrho$ is the outward unit normal vector of $\partial M$ in $(M, g)$.

The \emph{momentum tensor} $\pi$ of $(M, g, K)$ is given by
\[
\pi = K - (\operatorname{tr}K)g.
\]

We define $(\iota_{\varrho}\pi)^{\top}$ as the restriction of $\iota_{\varrho}\pi = \pi(\varrho, \cdot)$ to the tangent vector fields on $\partial M$. The initial data set $(M, g, K)$ is said to satisfy the \emph{boundary dominant energy condition} (BDEC) if
\[
H^{\partial M} \geq |(\iota_{\varrho}\pi)^{\top}| \quad \text{on} \quad \partial M.
\]

The BDEC was introduced by S.~Almaraz, L.~L.~de~Lima, and L.~Mari~\cite{AlmarazLimaMari2021} in the context of proving positive mass theorems for asymptotically flat and asymptotically hyperbolic initial data sets with non-compact boundaries (see~\cite[Remark~2.7]{AlmarazLimaMari2021} for the motivation behind this definition).

Let $\Sigma$ be a two-sided hypersurface in $(M, g)$. Fix a field $N$ of unit normal vectors to $\Sigma$ in $(M, g)$. The \emph{null mean curvatures} $\theta^+$ and $\theta^-$ of $\Sigma$ in $(M, g, K)$ are defined as
\[
\theta^+ = \operatorname{tr}_{\Sigma} K + H^{\Sigma} 
\quad \text{and} \quad 
\theta^- = \operatorname{tr}_{\Sigma} K - H^{\Sigma},
\]
where $H^{\Sigma} = \div_\Sigma N$ is the mean curvature of $\Sigma$ in $(M, g)$.

Following R.~Penrose, we say that the hypersurface $\Sigma$ is \emph{outer trapped} if $\theta^+ < 0$, \emph{weakly outer trapped} if $\theta^+ \leq 0$, and \emph{marginally outer trapped} if $\theta^+ = 0$. In the latter case, $\Sigma$ is referred to as a \emph{MOTS} (an abbreviation for \emph{marginally outer trapped surface}).

The \emph{null second fundamental forms} $\chi^+$ and $\chi^-$ of $\Sigma$ in $(M, g, K)$ are defined as
\[
\chi^+ = K|_{\Sigma} + A 
\quad \text{and} \quad 
\chi^- = K|_{\Sigma} - A,
\]
where $A$ is the second fundamental form of $\Sigma$ in $(M, g)$, given by
\[
A(Y, Z) = \langle \nabla_Y N, Z \rangle, \quad Y, Z \in \mathfrak{X}(\Sigma).
\]
Note that $\theta^\pm = \operatorname{tr} \chi^\pm$.

Let $(\Sigma_{t})_{|t|<\varepsilon}$ be a variation of $\Sigma$ in $M$, with $\Sigma = \Sigma_0$ and variational vector field
\[
\mathcal{V} = \frac{\partial}{\partial t}\Big|_{t=0} = \phi N \quad \text{for some} \quad \phi \in C^{\infty}(\Sigma).
\]

We can view the null mean curvature $\theta^+ = \theta^+(t)$ of $\Sigma_t$ in $(M, g, K)$ as a one-parameter family of functions defined on $\Sigma$. It is known (see~\cite{AnderssonMarsSimon2008}) that
\begin{equation}\label{eq:variacaotheta}
\frac{\partial \theta^+}{\partial t}\Big|_{t=0} = L\phi + \Big(-\frac{1}{2} (\theta^+)^2 + \theta^+ \tau \Big) \phi,
\end{equation}
where
\[
L\phi = -\Delta \phi + 2\langle X, \nabla \phi \rangle + ( Q - |X|^2 + \operatorname{div} X ) \phi 
\]
and
\[
Q = \frac{1}{2} R^{\Sigma} - (\mu + J(N)) - \frac{1}{2} |\chi^+|^2.
\]

Here, $R^{\Sigma}$ is the scalar curvature of $\Sigma$ with respect to the induced metric. Additionally, $X$ is the vector field tangent to $\Sigma$ that is dual to the 1-form $K(N, \cdot)|_\Sigma$, and $\tau = \operatorname{tr} K$.

The above operator $L$ is referred to as the \emph{stability operator for MOTS}. This terminology arises because, in the Riemannian case (\emph{i.e.}\ when $K \equiv 0$), a MOTS corresponds to a minimal hypersurface, and the operator $L$ reduces to the classical stability operator (or the Jacobi operator) in the theory of minimal surfaces.

It is worth noting that the notion of stability for closed MOTS was introduced by\linebreak L.~Andersson, M.~Mars, and W.~Simon~\cite{AnderssonMarsSimon2008}. In the case of capillary MOTS in initial data sets with boundary, the notion of stability was introduced by A.~Alaee, M.~Lesourd, and S.-T.~Yau~\cite{AlaeeLesourdYau2021}. In this work, we will only consider the case of compact free boundary MOTS in initial data sets $(M, g, K)$ with boundary.

From now on, we assume that $M$ is a differentiable manifold with boundary and that $\Sigma$ is a properly embedded, compact hypersurface with boundary in $M$. That is, $\Sigma$ is embedded in $M$ and satisfies $\partial \Sigma = \Sigma \cap \partial M$.

We say that $\Sigma$ is \emph{free boundary} in $(M, g)$ if $\Sigma$ intersects $\partial M$ orthogonally, \emph{i.e.}\ $\varrho = \nu$ along $\partial \Sigma$, where $\nu$ is the outward unit normal vector of $\partial \Sigma$ in $\Sigma$ with respect to the induced metric.

A compact free boundary MOTS $\Sigma$ in $(M, g, K)$ is said to be \emph{stable} (see Definition~5.1 in~\cite{AlaeeLesourdYau2021}) if there exists a non-negative function $\phi \in C^{\infty}(\Sigma) \setminus \{0\}$ satisfying
\[
\begin{cases}
L\phi = -\Delta \phi + 2 \langle X, \nabla \phi \rangle + ( Q - |X|^2 + \operatorname{div} X ) \phi \geq 0 & \text{on} \quad \Sigma,\vspace{0.1cm}\\
B\phi = \dfrac{\partial \phi}{\partial \nu} - \mathrm{II}^{\partial M}(N, N) \phi = 0 & \text{on} \quad \partial \Sigma.
\end{cases}
\]
Without loss of generality, by the maximum principle for non-negative functions, we can assume that $\phi > 0$.

Suppose $\Sigma$ is a MOTS that separates $(M, g, K)$, that is, $M \setminus \Sigma$ is disconnected. The \emph{exterior} of $\Sigma$, denoted by $M_+$, is the region consisting of $\Sigma$ and the portion of $M \setminus \Sigma$ toward which the unit normal vector $N$ points.

We say that $\Sigma$ is \emph{outermost} if no hypersurface homologous to $\Sigma$, contained in $M_+$, except $\Sigma$, has null mean curvature $\theta^+ \leq 0$. Equivalently, $\Sigma$ is outermost if, for any hypersurface $\Sigma' \neq \Sigma$ homologous to $\Sigma$ and contained in $M_+$, there exists a point $p \in \Sigma'$ such that $\theta^+(p) > 0$. Similarly, we say that $\Sigma$ is \emph{weakly outermost} if no hypersurface homologous to $\Sigma$, contained in $M_+$, has null mean curvature $\theta^+ < 0$.

We say that $K$ is \emph{$n$-convex} on $M_+$ if $\operatorname{tr}_\pi K \geq 0$ for any $p \in M_+$ and any $n$-dimensional linear subspace $\pi \subset T_p M$. Equivalently, $K$ is $n$-convex if the sum of the $n$ smallest eigenvalues of $K$ is always non-negative.

The next three results are due to the second-named author~\cite{Mendes2022}, who generalized results by G.~J.~Galloway and R.~Schoen~\cite{GallowaySchoen2006} to the context of compact free boundary MOTS in initial data sets with boundary.

\begin{lem}[Mendes, 2022]\label{lem:Mendes1}
Let $(\Sigma^n, h)$, with $n \geq 2$, be an orientable, connected, compact Riemannian manifold with boundary. Suppose there exists a function \( u \in C^{\infty}(\Sigma) \), \( u > 0 \), such that 
\[
\begin{cases} 
\mathcal{L} u = -\Delta u + (\frac{1}{2} R^{\Sigma} - P) u \geq 0 & \text{on}\quad \Sigma,\vspace{0.1cm} \\ 
\dfrac{\partial u}{\partial \nu} + H^{\partial \Sigma} u \geq 0 & \text{on}\quad \partial \Sigma, 
\end{cases} 
\]
where \( R^{\Sigma} \) is the scalar curvature of \( (\Sigma, h) \), \( H^{\partial \Sigma} \) is the mean curvature of \( \partial \Sigma \) in \( (\Sigma, h) \), and \( P \) is a non-negative function on \( \Sigma \). Then one of the following conditions holds:
\begin{enumerate} 
\item[$(1)$] \( \Sigma \) admits a metric with positive scalar curvature and minimal boundary; 
\item[$(2)$] \( (\Sigma, h) \) is Ricci-flat with a totally geodesic boundary, \( P \equiv 0 \), and \( u \) is constant. 
\end{enumerate} 
\end{lem}

Conclusion~(2) of Lemma~\ref{lem:Mendes1} will be important for our purposes. In fact, one of the hypotheses of Theorem~\ref{theo:A} states that $\Sigma$ does not satisfy condition~(1). Thus, by verifying that the assumptions of Lemma~\ref{lem:Mendes1} are met, we ensure the validity of condition~(2).

\begin{lem}[Mendes, 2022]\label{lem:Mendes3}
Let \( (M^{n+1}, g, K) \), with \( n \ge 2 \), be an initial data set with boundary, and let \( \Sigma \) be a compact free boundary MOTS in \( (M, g, K) \). If \( \Sigma \) is stable, then the first eigenvalue \( \lambda_1(\mathcal{L}_0) \) of the operator \( \mathcal{L}_0 = -\Delta + Q \) on \( \Sigma \) with Robin boundary condition 
\[
B_0 u = B u + \langle X, \nu \rangle u = 0 \quad \text{on} \quad \partial \Sigma
\]
is non-negative.
\end{lem}

It is well known that the first eigenvalue \( \lambda_1(\mathcal{L}_0) \) of the operator \( \mathcal{L}_0 \) defined in the previous lemma can be characterized as follows:
\[
\lambda_1(\mathcal{L}_0) = \inf_{u \in C^{\infty}(\Sigma) \setminus \{0\}} \frac{\ds\int_{\Sigma} ( |\nabla u|^2 + Q u^2 ) dv - \int_{\partial \Sigma}( \mathrm{II}^{\partial M}(N, N) - \langle X, \nu \rangle ) u^2 ds}{\ds\int_{\Sigma} u^2 dv}.
\]

\begin{theo}[Mendes, 2022]\label{teorema:Mendes}
Let \((M^{n+1}, g, K)\), with \(n \geq 2\), be an initial data set with boundary, and let \(\Sigma\) be a compact, free boundary stable MOTS in \((M, g, K)\). Assume that \((M, g, K)\) satisfies both the DEC and the BDEC, \emph{i.e.}\ \(\mu \geq |J|\) on \(M\) and \(H^{\partial M} \geq |(\iota_{\varrho} \pi)^{\top}|\) on~\(\partial M\). If \(\Sigma\) is weakly outermost in \((M, g, K)\) and does not admit a metric with positive scalar curvature and minimal boundary, then there exists an outer neighborhood \(V \cong [0, \delta) \times \Sigma\) of \(\Sigma\) in \(M\) such that
\[
g|_V = \varphi^2 d t^2 + h_t,
\]
where \(\varphi: V \to \mathbb{R}\) is a positive function and \(h_t\) is the induced metric on \(\Sigma_t \cong \{t\} \times \Sigma\). Moreover, the following conditions hold:
\begin{enumerate}
 \item[$(1)$] \(\Sigma_t \cong \{t\} \times \Sigma\) is a free boundary MOTS, and it has vanishing outward null second fundamental form;
 \item[$(2)$] \(\Sigma_t\) is Ricci-flat and has a totally geodesic boundary with respect to the induced metric;
 \item[$(3)$] The DEC is saturated on \(V\), and \(J|_{\Sigma_t} = 0\);
 \item[$(4)$] The BDEC is saturated on \(\partial \Sigma_t\), and \((\iota_{\varrho} \pi)^{\top}|_{\partial \Sigma_t} = 0\).
\end{enumerate}
\end{theo}

A crucial step for Galloway and Jang in proving Theorem~\ref{theo:Galloway-Jang} was demonstrating that $\Sigma$ is a MOTS in the special initial data set $(M, g, K = -\epsilon g)$. To achieve this, they relied on a well-known result by L.~Andersson and J.~Metzger (see~\cite[Lemma~5.2]{AnderssonMetzger2009}).

In our context, a similar result is required for compact free boundary hypersurfaces. This will be established in Lemma~\ref{lem:K=-eg} later. However, before proceeding, we will state a maximum principle for parabolic equations, which plays a key role in the proof of the lemma.

Let $u: \Sigma \times [0, T] \to \mathbb{R}$, with $T > 0$, be a function, and let $\mathcal{L}u = \frac{\partial u}{\partial t} + Pu$ be a parabolic operator, with $P$ defined as 
\[
Pu = -a^{ij}(x, t)\nabla_i\nabla_j u - b^i(x, t)\nabla_i u - c(x, t)u,
\]
where $a^{ij}, b^i, c \in L^\infty(\Sigma \times [0, T])$.

Additionally, let $f:\mathbb{R} \times \Sigma \times [0, T] \to \mathbb{R}$ be a function satisfying the following conditions: 
\[
f(0, x, t) \geq 0,
\]
\[
|f(s, x, t) - f(0, x, t)| \leq \kappa(x, t) \cdot |s| \quad \text{for} \quad |s| \leq \rho,
\]
where $\kappa \in L^\infty(\Sigma \times [0, T])$ and $\rho > 0$ is a constant.

Now, consider a vector field $v$ defined on $\partial \Sigma \times [0, T]$, where $v$ is tangent to $\Sigma$ at each point and satisfies $g(v, \nu) > 0$, with $\nu$ denoting the outward unit normal field of $\partial \Sigma$ in $\Sigma$.

We are now ready to state the maximum principle:

\begin{theo}[Maximum principle for parabolic equations]\label{theo:PrincipioMaximo}
Let $\mathcal{L}$, $f$, and $v$ be as defined above, and let $u: \Sigma \times [0, T] \to \mathbb{R}$ be a continuous function such that $u$ is of class $C^{2,1}$ in a neighborhood of each point $(x, t) \in \Sigma \times [0, T]$ where $|u(x, t)| < \varepsilon$, with $\varepsilon > 0$ being a small constant. Suppose that $u$ satisfies the following conditions:
\begin{enumerate}
 \item[$(i)$] $u(\cdot, 0) \geq 0$;
 \item[$(ii)$] $\mathcal{L}u(x, t) \geq 0$ for every $(x, t) \in \Sigma \times (0, T]$ with $|u(x, t)| < \rho$;
 \item[$(iii)$] $\frac{\partial u}{\partial v}(x, t) = f(u(x, t), x, t)$ for each $(x, t) \in \partial \Sigma \times (0, T]$ with $|u(x, t)| < \rho$.
\end{enumerate}
Then, $u \geq 0$. Furthermore, if $u(\cdot, 0)$ is not identically zero, then $u(x, t) > 0$ for every $(x, t) \in \Sigma \times (0, T) \cup \operatorname{int}(\Sigma) \times \{T\}$.
\end{theo}

This result is discussed in detail in~\cite{Stahl1996}.

\begin{rmk}
It is important to mention that the maximum principle stated above will be applied in a very specific setting. In this case:
\begin{itemize}
 \item $u$ will be of class $C^\infty$ on $\Sigma \times [0, T]$,
 \item $\mathcal{L}u = 0$ on $\Sigma \times [0, T]$,
 \item $\frac{\partial u}{\partial v}(x, t) = f(u(x, t), x, t)$ for every $(x, t) \in \partial \Sigma \times [0, T]$, and
 \item the function $f$ will satisfy
 \[
 |f(s, x, t) - f(0, x, t)| = \kappa(x, t) \cdot |s|, \quad \forall (s, x, t) \in \mathbb{R} \times \Sigma \times [0, T].
 \]
\end{itemize}
In this special case, the constants $\rho > 0$ and $\varepsilon > 0$ are unnecessary.
\end{rmk}

Next, we present the definition of the mean curvature flow with Neumann boundary\linebreak (or free boundary) condition.

Consider a hypersurface $S$ in $\mathbb{R}^{n+1}$ and an orientable compact manifold $\Sigma^n$ with boundary. Let $F_0: \Sigma \to \mathbb{R}^{n+1}$ be an immersion such that $\Sigma_0 := F_{0}(\Sigma)$ satisfies the following conditions:
\[
\begin{cases}
\partial \Sigma_0 = F_{0}(\partial \Sigma) = \Sigma_0 \cap S, \\
\langle N_0, \varrho \circ F_0 \rangle(x) = 0, \quad \forall x \in \partial \Sigma,
\end{cases}
\]
where $N_0$ and $\varrho$ are unit normal fields to $\Sigma_0$ and $S$, respectively.

Let $F: \Sigma^n \times [0, T) \to \mathbb{R}^{n+1}$ be a family of immersions with $F(\cdot, 0) = F_0$. We say that $F$ evolves by the \emph{free boundary mean curvature flow} if it satisfies the following conditions:
\begin{equation}\label{eq:fluxo}
\begin{cases}
\dfrac{d}{dt}F(x, t) = \vec{H}(x, t), & \forall (x, t) \in \Sigma \times [0, T),\vspace{0.1cm} \\
F(\partial \Sigma, t) \subset S, & \forall t \in [0, T),\vspace{0.1cm} \\
\langle N, \varrho \circ F \rangle(x, t) = 0, & \forall (x, t) \in \partial \Sigma \times [0, T),
\end{cases}
\end{equation}
where $\vec{H}$ denotes the mean curvature vector of the immersion.

The short-time existence of the free boundary mean curvature flow was established by A.~Stahl~\cite{Stahl1996}, inspired by the work of K.~Ecker and G.~Huisken~\cite{EckerHuisken1991}, who proved the short-time existence of the mean curvature flow in the case where \( \Sigma_0 \) is closed. Similarly, it can be shown that the free boundary mean curvature flow exists for a small time \( T > 0 \) in Riemannian manifolds \( (M^{n+1}, g) \) with boundary \( \partial M \). In this general setting, $S=\partial M$ and $\Sigma_0$ is a compact free boundary hypersurface in $M$.

To apply the maximum principle for parabolic equations (Theorem~\ref{theo:PrincipioMaximo}) to a function \( u \), it is necessary to verify that \( u \) satisfies conditions \( (i) \), \( (ii) \), and \( (iii) \). The following result will assist in establishing condition \( (iii) \).

Consider a variation \( f: \Sigma \times (-\delta, \delta) \to M \) of \( \Sigma \) such that, for each \( t \in (-\delta, \delta) \), the map \( f_t: \Sigma \ni x \mapsto f(x, t) \in M \) is an immersion satisfying \( f_t(\partial \Sigma) \subset \partial M \). We denote by \( \Sigma_t \), \( N_t \), and \( H^{\Sigma_t} \) the image \( f_t(\Sigma) \), a unit normal vector field to \( \Sigma_t \), and the mean curvature of \( \Sigma_t \)\linebreak (\emph{i.e.}\ \( H^{\Sigma_t} = \div_{\Sigma_t} N_t \)), respectively.

The variational vector field of \( f \) is expressed as 
\[ 
\partial_t = \frac{\partial f}{\partial t}. 
\] 
Decomposing \( \partial_t \) into its tangential and normal components yields:
\[ 
\partial_t = \partial_t^T + \phi_t N_t, 
\] 
where \( \phi_t \) is a function on \( \Sigma_t \) defined by \( \phi_t = g(\partial_t, N_t) \), with \( g \) representing the Riemannian metric on \( M \).

In the special case where the variational vector field has only a normal component, \emph{i.e.}\ \( \partial_t^T = 0 \), the result below describes how to compute the derivative of \( \phi_t \) in the direction of \( \nu_t \) (see~\cite[Proposition~17]{Ambrozio2015}).

\begin{prop}[Ambrozio, 2015]\label{prop:Ambrozio}
If \( \Sigma_0 \) is free boundary and \( \partial_t^T = 0 \) at \( t = 0 \), then 
\[ 
\partial_t g(N_t, \varrho) \big|_{t=0} = -\frac{\partial \phi_0}{\partial \nu_0} + g(N_0, \nabla_{N_0} \varrho) \phi_0, 
\] 
where \( \nu_0 \) is the outward unit normal of \( \partial \Sigma_0 \) in $\Sigma_0$.
\end{prop}

\subsection{The Yamabe invariant}\label{subsec:preliminares2}

In this subsection, we present some preliminaries for\linebreak Theorem~\ref{theo:B}.

The classical Yamabe problem, initially studied by H.~Yamabe~\cite{Yamabe1960} and later resolved through the contributions of N.~S.~Trudinger~\cite{Trudinger1968}, T.~Aubin~\cite{Aubin1976}, and R.~Schoen~\cite{Schoen1984}, asserts that for any Riemannian metric \( g \) on a closed manifold \( \Sigma^n \), with \( n \geq 3 \), there exists a metric \( \bar{g} \) in the conformal class of \( g \) with constant scalar curvature.

For compact manifolds \( \Sigma^n \) with boundary, two natural extensions of the Yamabe problem are as follows:
\begin{enumerate}
 \item[\text{(a)}] Given a Riemannian metric \( g \) on \( \Sigma \), does there exist a metric \( \bar{g} \) in the conformal class of \( g \) with constant scalar curvature and minimal boundary?
 \item[\text{(b)}] Given a Riemannian metric \( g \) on \( \Sigma \), does there exist a metric \( \bar{g} \) in the conformal class of \( g \) with zero scalar curvature and constant mean curvature on the boundary?
\end{enumerate}

These problems were introduced and first studied by J.~F.~Escobar~\cite{Escobar1992v2, Escobar1992}. Significant progress was later made by various authors, culminating in their complete resolution\linebreak (see~\cite{BrendleChen2014, Chen2010, Almaraz2010, Marques2005, Marques2007}).

It can be shown that \( \bar{g} = \varphi^{\frac{4}{n-2}} g \), where \( \varphi \in C^\infty(\Sigma) \) and \( \varphi > 0 \), solves the Yamabe problem with boundary (a) (resp.\ (b)) if and only if \( \varphi \) is a critical point of the functional
\[
Q_g^{a,b}(\varphi) = \frac{\ds\int_{\Sigma}\Big( \frac{4(n-1)}{n-2}|\nabla \varphi|^2 + R_g^\Sigma \varphi^2\Big)dv + 2 \int_{\partial \Sigma} H_g^{\partial \Sigma} \varphi^2 ds}{\left( a \ds\int_{\Sigma} \varphi^{\frac{2n}{n-2}} dv + b \big( \int_{\partial \Sigma} \varphi^{\frac{2(n-1)}{n-2}} ds \big)^{\frac{n}{n-1}} \right)^{\frac{n-2}{n}}}
\]
with \( (a,b) = (1,0) \) (resp.\ \( (a,b) = (0,1) \)). Here, \( R_g^\Sigma \) is the scalar curvature of \( (\Sigma,g) \), \( H_g^{\partial \Sigma} \) is the mean curvature of \( \partial \Sigma \), and \( dv \), \( ds \) denote the volume elements of \( \Sigma \) and \( \partial \Sigma \), respectively.

The \emph{Yamabe constant} \( Q_g^{a,b}(\Sigma, \partial\Sigma) \) of \( (\Sigma, g) \) is defined as
\[
Q_g^{a,b}(\Sigma, \partial\Sigma) = \inf_{\varphi \in C^\infty(\Sigma), \, \varphi > 0} Q_g^{a,b}(\varphi).
\]

We say that \( \bar{g} = \varphi^{\frac{4}{n-2}} g \) is a \emph{Yamabe metric} if \( \varphi \) attains the Yamabe constant, that is,
\[
Q_g^{a,b}(\varphi) = Q_g^{a,b}(\Sigma, \partial\Sigma).
\]

The \emph{conformal class} of \( g \) is denoted by
\[
[g] = \{ \varphi^{\frac{4}{n-2}} g \mid \varphi \in C^\infty(\Sigma), \varphi > 0 \},
\]
and the set of all conformal classes of metrics on \( \Sigma \) is written as \( \mathcal{C}(\Sigma) \). It is straightforward to verify that \( Q_g^{a,b}(\Sigma, \partial\Sigma) \) is a conformal invariant, meaning that if \( \bar{g} \in [g] \), then
\[
Q_{\bar{g}}^{a,b}(\Sigma, \partial\Sigma) = Q_g^{a,b}(\Sigma, \partial\Sigma).
\]
Thus, \( Q_g^{a,b}(\Sigma, \partial\Sigma) \) is often denoted by \( Q_{[g]}^{a,b}(\Sigma, \partial\Sigma) \).

The \emph{Yamabe invariant} \( \sigma^{a,b}(\Sigma, \partial\Sigma) \) of \( \Sigma \) is defined as the supremum of the Yamabe constants over all conformal classes:
\[
\sigma^{a,b}(\Sigma, \partial\Sigma) := \sup_{[g] \in \mathcal{C}(\Sigma)} Q_{[g]}^{a,b}(\Sigma, \partial\Sigma).
\]

Two important results are now presented for use in the proof of Theorem~\ref{theo:B}. The first is a uniqueness result due to Escobar~\cite{Escobar2003}:

\begin{theo}[Escobar, 2003]\label{theo:Escobar}
Let \( (\Sigma^n, g) \), with \( n \geq 2 \), be a compact Riemannian manifold with boundary. If \( \bar{g} \in [g] \) satisfies \( R^\Sigma_{\bar{g}} = R^\Sigma_g \leq 0 \) and \( H_{\bar{g}}^{\partial \Sigma} = H_g^{\partial \Sigma} \leq 0 \), then \( \bar{g} = g \).
\end{theo}

The second result, due to T.~Cruz and A.~S.~Santos~\cite[Proposition~5.3]{CruzSantos2023}, establishes conditions under which Yamabe metrics on compact manifolds with boundary are Einstein:

\begin{prop}[Cruz-Santos, 2023]\label{prop:CruzSantos}
Let \( \Sigma^n \), with \( n \geq 3 \), be a compact manifold with boundary such that the Yamabe invariant \( \sigma^{1,0}(\Sigma, \partial\Sigma) \) is negative. Suppose \( g \) is a metric on \( \Sigma \) that achieves the Yamabe invariant, that is,
\[
Q_g^{1,0}(\Sigma, \partial\Sigma) = \sigma^{1,0}(\Sigma, \partial\Sigma).
\]
Then every Yamabe metric in \( [g] \) is Einstein with a totally geodesic boundary.
\end{prop}

To conclude this subsection, we state a lemma due to the second-named author~\cite{Mendes2019TAMS}. This result, inspired by the techniques of Micallef and Moraru~\cite{MicallefMoraru2015} and Moraru~\cite{Moraru2016}, will also be used in the proof of Theorem~\ref{theo:B}.

\begin{lem}[Mendes, 2019]\label{lem:Mendes4}
Let \( f \in C^1([0, \delta)) \) and \( \eta, \xi, \rho \in C^0([0, \delta)) \) be functions satisfying \( \max\{f, \rho\} \geq 0 \), \( \xi \geq 0 \), \( \eta > 0 \), and \( f(0) = 0 \). If 
\[ 
f'(t) \eta(t) - f(t) \rho(t) \leq \int_0^t f(s) \xi(s) ds, \quad \forall t \in [0, \delta), 
\] 
then \( f \leq 0 \). In particular, if \( f \) is non-negative, then \( f \equiv 0 \).
\end{lem}

\section{Proof of Theorem~\ref{theo:A}}\label{sec:proof.Theo.A}

In this section, we present the proof of Theorem~\ref{theo:A}, inspired by the work of G.~J.~Galloway and H.~C.~Jang~\cite{GallowayJang2020}. Before proceeding, however, we will establish an auxiliary result analogous to Lemma~5.2 in~\cite{AnderssonMetzger2009}, which Galloway and Jang utilized in their proof of Theorem~\ref{theo:Galloway-Jang}:

\begin{lem}\label{lem:K=-eg} 
Let $(M^{n+1},g)$, with $n\ge2$, be a Riemannian manifold with boundary, and let~$\Sigma$ be a compact free boundary hypersurface in $(M,g)$ whose mean curvature satisfies $H^\Sigma\le n\epsilon$ and $H^\Sigma\not\equiv n\epsilon$, where $\epsilon=0$ or $1$ is fixed. Then, for any $r>0$, there exists a compact free boundary hypersurface $\Sigma'$ in $(M,g)$ that lies within an outer $r$-neighborhood of $\Sigma$ in $M$, satisfies $H^{\Sigma'}< n\epsilon$, and is disjoint from $\Sigma$, \emph{i.e.}\ $\Sigma\cap\Sigma'=\emptyset$. Moreover, $\Sigma'$ is a graph over $\Sigma$.
\end{lem}

\begin{proof}
Consider the tensor \( K = -\epsilon g \) and the \emph{free boundary null mean curvature flow}\linebreak \( F: \Sigma \times [0, T) \to (M, g, K) \) defined by 
\begin{equation}\label{eq:fluxotheta} 
\begin{cases}
\dfrac{d}{dt}F = -\theta^+ N& \text{on} \quad \Sigma \times [0, T),\vspace{0.1cm}
\\ 
F(\partial\Sigma, t) \subset \partial M,& \forall t \in [0, T),\vspace{0.1cm}
\\ 
\langle N, \varrho \circ F \rangle = 0& \text{on} \quad \partial\Sigma \times [0, T). 
\end{cases} 
\end{equation} 

Thus, for \(\epsilon = 0\), the flow \eqref{eq:fluxotheta} takes the form of the flow~\eqref{eq:fluxo}, which, as mentioned earlier, has a solution for a small time \(T > 0\). In the case \(\epsilon = 1\), the proof of the existence of \eqref{eq:fluxotheta} for small time can be carried out in a manner entirely analogous to the proof of the existence of \eqref{eq:fluxo}. 

To conclude the proof of the lemma, we will use the maximum principle for parabolic equations (Theorem~\ref{theo:PrincipioMaximo}) to ensure that \(\theta^+ < 0\) on \(\Sigma_t = F(\Sigma, t)\) for every \(t \in (0, T)\). 

First, it follows from \eqref{eq:variacaotheta} and \eqref{eq:fluxotheta} that 
\[
\frac{\partial \theta^+}{\partial t} = -L_t \theta^+ = \Delta \theta^+ - 2 \langle X, \nabla \theta^+ \rangle - \mathcal{Q} \theta^+,
\] 
where 
\[
\mathcal{Q} = \frac{1}{2} R^\Sigma - (\mu + J(N)) - \frac{1}{2} |\chi^+|^2 - |X|^2 + \operatorname{div} X - \frac{1}{2} (\theta^+)^2 + \theta^+ \tau,
\] 
with all geometric quantities computed on \(\Sigma_t\). 

Now, fix \(T_0 \in (0, T)\) and define \(u: \Sigma \times [0, T_0] \to \mathbb{R}\) by \(u = -e^{-t} \theta^+\). We aim to verify that \(u\) satisfies conditions \((i)\), \((ii)\), and \((iii)\) of Theorem~\ref{theo:PrincipioMaximo}. Since \(H^\Sigma \leq n\epsilon\), it follows that \(\theta^+ = H^\Sigma - n\epsilon \leq 0\) on \(\Sigma\). Therefore, condition \((i)\) holds, \emph{i.e.}\ \(u(\cdot, 0) \geq 0\). 

To verify condition \((ii)\), observe that 
\[
\begin{aligned}
\Big(\frac{\partial}{\partial t} - \Delta\Big)u &= \frac{\partial}{\partial t}(-e^{-t} \theta^+) + \Delta(e^{-t}\theta^+) \\ 
&= e^{-t} \theta^+ - e^{-t} \frac{\partial \theta^+}{\partial t} + e^{-t}\Delta\theta^+ \\ 
&= -u - e^{-t}(\Delta \theta^+ - 2 \langle X, \nabla \theta^+ \rangle - \mathcal{Q} \theta^+) + e^{-t}\Delta\theta^+ \\ 
&= -2 \langle X, \nabla u \rangle - (\mathcal{Q} + 1)u.
\end{aligned}
\] 
Therefore, it follows that 
\[
\mathcal{L} u = \frac{\partial u}{\partial t} + Pu = 0,
\] 
where \(Pu = -\Delta u + 2 \langle X, \nabla u \rangle + (\mathcal{Q} + 1)u\). This shows that condition \((ii)\) is satisfied. 

Finally, it follows from Proposition~\ref{prop:Ambrozio} that \(\frac{\partial \theta^+}{\partial \nu} = \operatorname{II}^{\partial M}(N, N) \theta^+\) and, therefore, 
\[
\frac{\partial u}{\partial \nu} = \operatorname{II}^{\partial M}(N, N) u \quad \text{on} \quad \partial \Sigma_t.
\] 
Thus, defining \(f: \mathbb{R} \times \Sigma \times [0, T_0] \to \mathbb{R}\) by \(f(s, x, t) = \operatorname{II}^{\partial M}(N, N)s\), we have: 
\begin{enumerate} 
\item[(a)] \(f(0, x, t) = 0\); 
\item[(b)] \(|f(s, x, t) - f(0, x, t)| = |\operatorname{II}^{\partial M}(N, N)| \cdot |s|\). 
\end{enumerate} 
This ensures that condition \((iii)\) is also satisfied. Since \(u(\cdot, 0) \not\equiv 0\), because \(H^\Sigma \not\equiv n\epsilon\), it follows from Theorem~\ref{theo:PrincipioMaximo} that \(u > 0\) in \(\Sigma \times (0, T_0) \cup \operatorname{int} \Sigma \times \{T_0\}\). Since this holds for every \(T_0 \in (0, T)\), we have \(\theta^+ < 0\) in \(\Sigma_t\) for each \(t \in (0, T)\). 

Therefore, the flow \(F\) evolves the hypersurface \(\Sigma\) in the direction of \(N\). In particular, for sufficiently small \(t > 0\), \(\Sigma_t\) will be a properly embedded hypersurface contained within an \(r\)-neighborhood exterior to \(\Sigma\) and satisfying \(\Sigma \cap \Sigma_t = \emptyset\). Hence, we can take \(\Sigma' = \Sigma_t\).
\end{proof}

Before proceeding to the proof of Theorem~\ref{theo:A}, we present an important definition:

Let $(M^{n+1}, g)$, with $n \geq 2$, be a Riemannian manifold with boundary, and let $\Sigma^n \subset M^{n+1}$ be a compact, properly embedded hypersurface. Assume $H_0 \in \mathbb{R}$ is such that $H^\Sigma \leq H_0$. If $\Sigma$ separates $M$, we say that $\Sigma$ is \emph{weakly outermost with respect to $H_0$} if no hypersurface $\Sigma'$, homologous to $\Sigma$ and contained in the exterior $M_+$ of $\Sigma$, has mean curvature $H^{\Sigma'} < H_0$. Furthermore, $\Sigma$ is said to be \emph{locally weakly outermost with respect to $H_0$} if there exists a neighborhood $U$ of $\Sigma$ in $M$ such that $\Sigma$ is weakly outermost with respect to $H_0$ in $(U, g|_U)$.

This definition appears in~\cite{GallowayJang2020}, with the only difference being that, here, we emphasize the role of the constant $H_0$.

With this definition in place, we can proceed to the proof of our first theorem:

\begin{proof}[Proof of Theorem~\ref{theo:A}]
Consider the initial data set \((M,g,K)\) with \(K = -\epsilon g\). 

We assert that \((M,g,K)\) satisfies the DEC. Indeed, considering \(\{e_1, \ldots, e_{n+1}\}\) as a local orthonormal frame on \(M\), we have 
\[
\operatorname{tr}K = \sum_{i=1}^{n+1} K(e_i, e_i) = -(n+1)\epsilon,
\] 
and 
\[
|K|^2 = \sum_{i,j=1}^{n+1} K(e_i, e_j)^2 = (n+1)\epsilon^2 .
\] 
Thus, we deduce
\[
\begin{aligned}
\mu &= \frac{1}{2}(R^M + (\operatorname{tr}K)^2 - |K|^2) \\
&= \frac{1}{2}(R^M + (n+1)n\epsilon^2) \\
&= \frac{1}{2}(R^M + (n+1)n\epsilon),
\end{aligned}
\] 
where the last equality follows from the fact that \(\epsilon = 0\) or \(\epsilon = 1\). Then, by hypothesis~(1), we have \(\mu \geq 0\). Moreover, 
\[
J = \div(K - (\operatorname{tr}K)g) = \div(n\epsilon g) = 0.
\] 
Thus, \(\mu - |J| \geq 0\), and the dominant energy condition is satisfied.

On the other hand, since \(\pi = K - (\operatorname{tr}K)g = n\epsilon g\) and \((\iota_\varrho\pi)^\top = \pi(\varrho, \cdot)|_{\partial M}\), it follows that 
\[
\begin{aligned}
(\iota_\varrho\pi)^\top &= n\epsilon g(\varrho, \cdot)|_{\partial M} = 0,
\end{aligned}
\] 
given that \(\varrho\) is normal to the boundary of \(M\). Thus, the boundary dominant energy condition also holds, as \(H^{\partial M} \geq 0 = |(\iota_\varrho\pi)^\top|\). 

\textbf{Claim 1:} \(\Sigma\) is a MOTS that is locally weakly outermost in \((M, g, K = -\epsilon g)\). 

First, observe that the null mean curvature \(\theta^+\) of a hypersurface \(\Sigma'\) in \((M, g, K = -\epsilon g)\) is given by \(\theta^+ = H^{\Sigma'} - n\epsilon\). Hence, asserting that \(\Sigma\) is a MOTS that is locally weakly outermost in \((M, g, K = -\epsilon g)\) is equivalent to saying that \(H^\Sigma \equiv n\epsilon\) and that \(\Sigma\) is locally weakly outermost in \((M, g)\) with respect to \(H_0 = n\epsilon\). Therefore, we only need to verify that \(H^\Sigma = n\epsilon\) on \(\Sigma\), due to hypothesis (4) of the theorem. 

Suppose \(H^\Sigma \not\equiv n\epsilon\). Since \(H^\Sigma \leq n\epsilon\), it follows from Lemma~\ref{lem:K=-eg} that, given an outer neighborhood \(U\) of \(\Sigma\) in \(M\), there exists a hypersurface \(\Sigma' \subset U\), homologous to $\Sigma$ in $U$, such that \(H^{\Sigma'} < n\epsilon\), which contradicts the fact that \(\Sigma\) is locally weakly outermost with respect to \(H_0 = n\epsilon\). Therefore, \(H^\Sigma \equiv n\epsilon\).

\textbf{Claim 2:} $\Sigma$ is a stable MOTS in $(M, g, K = -\epsilon g)$.

We aim to prove the existence of a positive function $\phi \in C^\infty(\Sigma)$ such that 
\[
\begin{cases}
L\phi = -\Delta\phi + 2\langle X, \nabla\phi \rangle + (Q - |X|^2 + \operatorname{div} X)\phi \geq 0 & \text{on} \quad \Sigma, \vspace{.1cm}\\
B\phi = \dfrac{\partial\phi}{\partial\nu} - \mathrm{II}^{\partial M}(N, N)\phi = 0 & \text{on} \quad \partial\Sigma,
\end{cases}
\]
where 
\[
Q = \frac{1}{2}R^\Sigma - (\mu + J(N)) - \frac{1}{2}|\chi^+|^2.
\]

To prove this, observe that $K(N,\cdot)|_\Sigma = -\epsilon g(N,\cdot)|_\Sigma = 0$ (as $N$ is normal to $\Sigma$), so $X = 0$. Moreover, we previously established that $\mu = \frac{1}{2}(R^M + (n+1)n\epsilon)$ and $J = 0$. Hence, the stability operator for MOTS reduces to 
\[
L\phi = -\Delta\phi + Q\phi,
\]
where 
\[
Q = \frac{1}{2}(R^\Sigma - R^M - (n+1)n\epsilon - |\chi^+|^2).
\]
In particular, $L$ is symmetric. Thus, there exists a positive eigenfunction $\phi_1$ of $L$, associated with the first eigenvalue $\lambda_1(L)$, satisfying the Robin boundary condition: 
\[
\begin{cases}
L\phi_1 = \lambda_1(L)\phi_1 & \text{on} \quad \Sigma, \vspace{.1cm}\\
B\phi_1 = 0 & \text{on} \quad \partial\Sigma.
\end{cases}
\]

We claim that $\lambda_1(L) \geq 0$. Suppose, by contradiction, that $\lambda_1(L) < 0$. Considering a variation $(\Sigma_t)$ of $\Sigma$ in $M$ with variational vector field $\mathcal{V} = \phi_1 N$, it follows from~\eqref{eq:variacaotheta} that 
\[
(\theta^+)'(0) = L\phi_1 = \lambda_1(L)\phi_1 < 0,
\]
where $(\theta^+)'(t)$ denotes the derivative of $\theta^+(t)$ with respect to $t$. Hence, $\theta^+(t) < 0$ for sufficiently small $t > 0$, contradicting the fact that $\Sigma$ is a locally weakly outermost MOTS. This proves that $\lambda_1(L) \geq 0$, ensuring that $\Sigma$ is stable.

Now, under the assumptions of Theorem~\ref{teorema:Mendes}, it follows that there exists an outer\linebreak neighborhood $V \cong [0, \delta) \times \Sigma$ of $\Sigma$ in $M$ with coordinates $(t, x)$ in $V$ such that the metric $g$ can be expressed as 
\[
g = \varphi^2 dt^2 + h_{ij} dx^i dx^j \quad \text{on} \quad V,
\]
where $\varphi = \varphi(t, x)$ is positive, $h_t = h_{ij}(t, x)dx^i dx^j$ is the induced metric on $\Sigma_t \cong \{t\} \times \Sigma$, and $\Sigma_t$ is a free boundary MOTS in $(M, g, K = -\epsilon g)$ for each $t \in [0, \delta)$. Moreover, it follows from Proposition~\ref{prop:Ambrozio} that 
\[
\frac{\partial\varphi}{\partial\nu_t} = \operatorname{II}^{\partial M}(N_t, N_t)\varphi \quad \text{on} \quad \partial\Sigma_t,
\]
where $\nu_t$ is the outward unit normal to $\partial\Sigma_t$ in $(\Sigma_t, h_t)$.

\textbf{Claim 3:} $\varphi$ is constant on each $\Sigma_t$. 

Indeed, since $\Sigma_t$ is a MOTS for every \( t \in [0, \delta) \) and \( \partial_t = \varphi N_t \) is the variational vector field of the family \( (\Sigma_t) \), it follows from~\eqref{eq:variacaotheta} that 
$$
0 = \frac{\partial \theta^+}{\partial t} = L_t \varphi = -\Delta \varphi + 2\langle X_t, \nabla \varphi \rangle + (Q_t - |X_t|^2 + \operatorname{div} X_t) \varphi,
$$ 
where 
$$
Q_t = \frac{1}{2} R^{\Sigma_t} - (\mu + J(N_t)) - \frac{1}{2} |\chi_t^+|^2.
$$ 

On the other hand, since \( K(N_t, \cdot)|_{\Sigma_t} = -\epsilon g(N_t, \cdot)|_{\Sigma_t} = 0 \), we have \( X_t = 0 \). Moreover, as noted above, \( \mu = \frac{1}{2}(R^M + \epsilon(n+1)n) \) and \( J = 0 \). Therefore, 
$$
-\Delta \varphi + \Big(\frac{1}{2} R^{\Sigma_t} - P_t\Big) \varphi = 0 \quad \text{on} \quad \Sigma_t,
$$ 
where 
\begin{align*}
P_t := \frac{1}{2}(R^M + \epsilon(n+1)n + |\chi_t^+|^2) \geq 0.
\end{align*} 

Furthermore, note that the mean curvature \( H^{\partial M} \) of \( \partial M \) in \( (M, g) \) along \( \partial \Sigma_t \) can be expressed as 
\begin{align*}
H^{\partial M} = H^{\partial \Sigma_t} + \operatorname{II}^{\partial M}(N_t, N_t),
\end{align*} 
where \( H^{\partial \Sigma_t} \) is the mean curvature of \( \partial \Sigma_t \) in \( (\Sigma_t, h_t) \). This follows because \( \Sigma_t \) is free boundary in \( (M, g) \) (see the proof of Proposition~\ref{prop:aux1} in the next section). Thus, along \( \partial \Sigma_t \), we have 
\begin{eqnarray*}
\frac{\partial \varphi}{\partial \nu_t} + H^{\partial \Sigma_t} \varphi &=& \operatorname{II}^{\partial M}(N_t, N_t) \varphi + H^{\partial \Sigma_t} \varphi \\ 
&=& H^{\partial M} \varphi \\ 
&\geq& 0.
\end{eqnarray*} 

It then follows from Lemma~\ref{lem:Mendes1} that \( \varphi \) is constant on \( \Sigma_t \), \( P_t \equiv 0 \), and \( (\Sigma_t, h_t) \) is Ricci-flat with a totally geodesic boundary for each \( t \in [0, \delta) \). 

By making a change of variable if necessary, we can assume, without loss of generality, that \( \varphi \equiv 1 \). Thus, we have 
$$
g = dt^2 + h_{ij} dx^i dx^j \quad \text{on} \quad V.
$$ 

On the other hand, since \( P_t \equiv 0 \), we have \( \chi_t^+ = 0 \), \emph{i.e.}\ 
$$
\chi_t^+ = K|_{\Sigma_t} + A_t = -\epsilon h_t + A_t = 0,
$$ 
where \( A_t \) is the second fundamental form of \( \Sigma_t \) in \( (M, g) \). Therefore, \( A_t = \epsilon h_t \), which in coordinates means that 
$$
\frac{\partial h_{ij}}{\partial t} = 2\epsilon h_{ij}.
$$ 

Thus, 
$$
h_{ij}(t, x) = e^{2\epsilon t} h_{ij}(0, x).
$$ 

Finally, taking \( h = h_0 = g|_\Sigma \), it follows that \( g|_V = dt^2 + e^{2\epsilon t} h \), where \( (\Sigma, h) \) is Ricci-flat with a totally geodesic boundary, as desired. 
\end{proof}

One consequence of Theorem~\ref{theo:A} is the following:

\begin{cor}\label{cor}
Under the assumptions of Theorem~\ref{theo:A}, if $(M,g)$ is complete, then $(M_+,g|_{M_+})$ is isometric to $([0,\infty)\times\Sigma,dt^2+e^{2\epsilon t}h)$, where $h=g|_\Sigma$ and $(\Sigma,h)$ is Ricci-flat with a totally geodesic boundary.
\end{cor}

In the preceding result, it suffices for $(M_+,g|_{M_+})$ to be complete. 

To prove Corollary~\ref{cor}, we can follow the same reasoning as in the proof of Theorem~3.1 in~\cite{GallowayJang2020}.

\begin{example}
Consider the $(n+1)$-dimensional manifold  
\[
N^{n+1} = \{x=(x_1,\ldots,x_{n+1})\in\mathbb{R}^{n+1}\mid r(x)\ge(m/2)^{\frac{1}{n-1}}\}
\]
equipped with the spatial Schwarzschild metric   
\[
g_{\text{Sch}} = \left(1 + \frac{m}{2r^{n-1}}\right)^{\frac{4}{n-1}} \delta,
\]
where $m>0$, \(\delta\) is the Euclidean metric, and \(r(x) = |x|\) denotes the Euclidean distance from \(x\in N\) to the origin. Standard computations show that the scalar curvature of \((N, g_{\text{Sch}})\) vanishes.  

Furthermore, the mean curvature of the boundary  
\(
S = \{ r = (\frac{m}{2})^{\frac{1}{n-1}} \}
\)  
of \(N\) is zero. More generally, the mean curvature of a sphere \(S' = \{ r = r_0 \}\) is given by  
\[
H^{S'} = \frac{n}{r_0} \left(1 - \frac{m}{r_0^{n-1} + \frac{m}{2}}\right).
\]  
In particular, we have \(H^{S'} > 0\) for \(r_0 > (\frac{m}{2})^{\frac{1}{n-1}}\). By the maximum principle, this implies that \(S\) is weakly outermost with respect to \(H_0 = 0\) (see also Proposition~2.2 in~\cite{GallowayMendes2024}).  

This confirms that \((N, g_{\text{Sch}})\) satisfies all the hypotheses of Theorem~\ref{theo:Galloway-Jang} for \(\epsilon = 0\), except for condition~(3). However, the conclusion of Theorem~\ref{theo:Galloway-Jang} does not hold.  

Now, consider an extension of this example. Let \(M^{n+1}\) be the manifold with boundary  
\[
M^{n+1} = \{ x = (x_1, \ldots, x_{n+1}) \in \mathbb{R}^{n+1} \setminus \{0\} \mid x_{n+1} \geq 0 \}
\]
equipped with the metric  
\[
g = \left(1 + \frac{m}{2r^{n-1}}\right)^{\frac{4}{n-1}} \delta.
\]  
The boundary \(\partial M\) of $M$ has zero mean curvature; in fact, \(\partial M\) is totally geodesic in \((M, g)\). Furthermore, the hemisphere  
\(
\Sigma = \{ r = (\frac{m}{2})^{\frac{1}{n-1}} \} \cap \{ x_{n+1} \geq 0 \}
\)
is a free boundary hypersurface in \((M, g)\). As before, all the hypotheses of Theorem~\ref{theo:A} are satisfied, except for condition~(3). This example demonstrates that condition~(3) in Theorem~\ref{theo:A} is necessary when \(\epsilon = 0\).
\end{example}

\begin{example}
Consider the $(n+1)$-dimensional manifold  
\(
N^{n+1} = [r_m,\infty) \times S^n
\)  
equipped with the AdS-Schwarzschild metric  
\[
g_{\text{AdS-Sch}} = V(r)^{-1} dr^2 + r^2 h_{S^n},
\]  
where  
\[
r_m = (2m)^{\frac{1}{n-1}}, \quad V(r) = 1 + r^2 - \frac{2m}{r^{n-1}},
\]  
and \( h_{S^n} \) is the standard round metric of constant curvature one on \( S^n \).

It follows that \((N, g_{\text{AdS-Sch}})\) has constant scalar curvature \( -(n+1)n \). Furthermore, the mean curvature of a slice \( \{r_0\}\times S^n \) is given by  
\[
H(r_0) = n \frac{V(r_0)^{\frac{1}{2}}}{r_0}.
\]  
In particular, the boundary  
\(
S = \{r_m\} \times S^n
\)  
of \( N \) has constant mean curvature \( H(r_m) = n \); and for \( r_0 > r_m \), we have \( H(r_0) > n \). Thus, \( S \) is weakly outermost with respect to \( H_0 = n \) (see Proposition~2.2 in~\cite{GallowayMendes2024}). This confirms that condition~(3) in Theorem~\ref{theo:Galloway-Jang}, as pointed out by Galloway and Jang~\cite[Remark~1]{GallowayJang2020}, is also necessary when \( \epsilon = 1 \).  

Now, we extend this example by considering the manifold with boundary  
\(
M = (r_+, \infty) \times S^n_+
\)  
equipped with the metric  
\[
g = V(r)^{-1} dr^2 + r^2 h_{S^n_+},
\]  
where \( S^n_+ \) is a hemisphere of \( S^n \), and \( h_{S^n_+} \) is the induced metric on \( S^n_+ \). Here, \( r_+ \) is the unique positive number such that \( V(r_+) = 0 \). This example satisfies all the hypotheses of Theorem~\ref{theo:A} for \( \epsilon = 1 \), except for condition~(3).
\end{example}

\begin{example}  
As a final example in this section, consider the manifold \( N^{n+1} = [r_0, \infty) \times T^n \) equipped with the Kottler metric  
\begin{align*}  
g_{\text{K}} = \left(r^2 - \frac{2m}{r^{n-1}}\right)^{-1} dr^2 + r^2 h,  
\end{align*}  
where \( r_0 > (2m)^{\frac{1}{n+1}} \) and \( h \) is a flat metric on \( T^n \). The scalar curvature of \( (N, g_{\text{K}}) \) is constant equals \( -(n+1)n \), and the mean curvature of a slice \( \{r\} \times T^n \) is given by  
\begin{align*}  
H(r) = n \frac{( r^2 - \frac{2m}{r^{n-1}})^{\frac{1}{2}}}{r} < n.  
\end{align*}  
Moreover, the boundary \( \{r_0\} \times T^n \) of \( N \) does not support a metric of positive scalar curvature. This example demonstrates that condition (4) in Theorem~\ref{theo:Galloway-Jang} is also necessary when \( \epsilon = 1 \).  

To extend this example to the case of free boundary hypersurfaces, consider the manifold  
\(
M^{n+1} = (r_+, \infty) \times [a, b] \times T^{n-1}  
\)  
endowed with the Riemannian metric  
\begin{align*}  
g = \left(r^2 - \frac{2m}{r^{n-1}}\right)^{-1} dr^2 + r^2 h_0,  
\end{align*}  
where \( r_+ = (2m)^{\frac{1}{n+1}} \), \( a < b \), and \( h_0 \) is a flat metric on \( [a, b] \times T^{n-1} \). This example satisfies all the hypotheses of Theorem~\ref{theo:A} for \( \epsilon = 1 \), except for condition (4).  
\end{example}

\section{Proof of Theorem~\ref{theo:B}}\label{sec:proof.Theo.B}

The goal of this section is to prove Theorem~\ref{theo:B}. However, before doing so, we will present and prove several auxiliary results. The first of these is a volume estimate for compact, free boundary stable MOTS with a negative Yamabe invariant:

\begin{prop}\label{prop:aux1} 
Let \((M^{n+1}, g, K)\), with \(n \geq 3\), be an initial data set with boundary, and let \(\Sigma^n\) be a compact, free boundary stable MOTS in \((M^{n+1}, g, K)\). 
\begin{enumerate} 
\item[$(1)$] If \(\mu + J(N) \geq -c\) on \(\Sigma\) for some constant \(c > 0\), \(H^{\partial M} \geq |(\iota_{\varrho} \pi)^{\top}|\) on \(\partial M\), and \(\Sigma\) has Yamabe invariant \(\sigma^{1,0}(\Sigma, \partial\Sigma) < 0\), then 
\begin{equation}\label{eq1:prop:aux1} 
\operatorname{vol}(\Sigma) \geq \left(\frac{|\sigma^{1,0}(\Sigma, \partial\Sigma)|}{2c}\right)^{\frac{n}{2}}. 
\end{equation} 
\item[$(2)$] If \(\mu + J(N) \geq 0\) on \(\Sigma\), \(H^{\partial M} - |(\iota_{\varrho} \pi)^{\top}| \geq -\Bar{c}\) on \(\partial M\) for some constant \(\Bar{c} > 0\), and \(\Sigma\) has Yamabe invariant \(\sigma^{0,1}(\Sigma, \partial\Sigma) < 0\), then 
$$ 
\operatorname{vol}(\partial\Sigma) \geq \left(\frac{|\sigma^{0,1}(\Sigma, \partial\Sigma)|}{2\Bar{c}}\right)^{n-1}. 
$$ 
\end{enumerate} 
\end{prop} 

The above proposition is the analog of Theorem~1 in~\cite{BarrosCruz2020} for compact free boundary MOTS in initial data sets with boundary.

\begin{proof}
We will begin with the proof of item~(1).

First, since \(\Sigma\) is a stable MOTS, it follows from Lemma~\ref{lem:Mendes3} that \(\lambda_1(\mathcal{L}_0) \geq 0\), \emph{i.e.}\ 
$$
0 \leq \lambda_1(\mathcal{L}_0) = \inf_{u \in C^{\infty}(\Sigma) \setminus \{0\}} \frac{\ds\int_{\Sigma} ( |\nabla u|^2 + Q u^2 ) dv - \int_{\partial \Sigma}( \mathrm{II}^{\partial M}(N, N) - \langle X, \nu \rangle ) u^2 ds}{\ds\int_{\Sigma} u^2 dv}.
$$
Thus, for \(u \in C^\infty(\Sigma)\), we have:
\begin{eqnarray}
0 &\leq & \int_{\Sigma}( 2 |\nabla u|^2 + 2 Q u^2 ) dv - 2 \int_{\partial \Sigma}( \operatorname{II}^{\partial M}(N, N) - \langle X, \nu \rangle ) u^2 ds \nonumber \\
&=& \int_{\Sigma}( 2 |\nabla u|^2 + ( R^{\Sigma} - 2 (\mu + J(N)) - |\chi^+|^2 ) u^2 ) dv \nonumber \\
&& - 2 \int_{\partial \Sigma}( \operatorname{II}^{\partial M}(N, N) - \langle X, \nu \rangle ) u^2 ds \nonumber \\
&\leq & \int_{\Sigma}( 2 |\nabla u|^2 + ( R^{\Sigma} + 2c ) u^2) dv - 2 \int_{\partial \Sigma} ( \operatorname{II}^{\partial M}(N, N) - \langle X, \nu \rangle) u^2 ds, \label{eq2:prop:aux1}
\end{eqnarray}
where we have used that \(\mu + J(N) \geq -c\) on \(\Sigma\).

On the other hand, since \(\Sigma\) is free boundary in \((M, g)\), it follows that the conormal \(\nu\) of \(\partial \Sigma\) in \(\Sigma\), with respect to the induced metric \(h = g|_\Sigma\), coincides with the conormal \(\varrho\) of \(\partial M\) in \((M, g)\) along \(\partial \Sigma\). Thus, at each point \(p \in \partial \Sigma\), we can choose an orthonormal frame \(\{e_1, \dots, e_{n-1}, e_n\}\) of \(T_p \partial M\) such that \(e_n = N\), \emph{i.e.}\ such that \(\{e_1, \dots, e_{n-1}\}\) is an orthonormal frame of \(T_p \partial \Sigma\). Therefore, 
\begin{align*}
H^{\partial M} &= \sum_{i=1}^{n-1} \langle \nabla_{e_i} \varrho, e_i \rangle + \langle \nabla_N \varrho, N \rangle \\
&= \sum_{i=1}^{n-1} \langle \nabla_{e_i} \nu, e_i \rangle + \operatorname{II}^{\partial M}(N, N) \\
&= H^{\partial \Sigma} + \operatorname{II}^{\partial M}(N, N);
\end{align*}
which, when substituted into~(\ref{eq2:prop:aux1}), gives:
\begin{equation}\label{eq3:prop:aux1}
0 \leq \int_{\Sigma} ( 2 |\nabla u|^2 + ( R^{\Sigma} + 2c ) u^2 ) dv - 2 \int_{\partial \Sigma} ( H^{\partial M} - H^{\partial \Sigma} - \langle X, \nu \rangle ) u^2 ds
\end{equation}
for every function \(u \in C^{\infty}(\Sigma)\). Therefore, defining \(a_n = \frac{4(n-1)}{n-2}\), it follows from~\eqref{eq3:prop:aux1}, together with Hölder's inequality, that
\begin{eqnarray}\label{eq4:prop:aux1}
0 &\leq & \int_{\Sigma} ( a_n |\nabla u|^2 + R^{\Sigma} u^2 ) dv + 2c \operatorname{vol}(\Sigma)^{\frac{2}{n}}\left(\int_{\Sigma} u^{\frac{2n}{n-2}}dv\right)^{\frac{n-2}{n}} \nonumber \\
&& + 2 \int_{\partial \Sigma} H^{\partial \Sigma} u^2 ds - 2 \int_{\partial \Sigma} ( H^{\partial M} - \langle X, \nu \rangle ) u^2 ds,
\end{eqnarray}
since \(a_n > 2\) for all \(n \geq 3\).

Along \(\partial \Sigma\), we have
$$
(\iota_{\varrho} \pi)^{\top}(N) = K(\varrho, N) - (\operatorname{tr} K) \langle \varrho, N \rangle = K(\nu, N) = \langle X, \nu \rangle,
$$
and
$$
H^{\partial M} \geq |(\iota_{\varrho} \pi)^{\top}| \geq (\iota_{\varrho} \pi)^{\top}(N),
$$
which implies that
$$
H^{\partial M} - \langle X, \nu \rangle = H^{\partial M} - (\iota_{\varrho} \pi)^{\top}(N) \geq 0.
$$
Thus, using the above inequality in~\eqref{eq4:prop:aux1}, we obtain
$$
0 \leq \frac{\ds\int_{\Sigma} ( a_n |\nabla u|^2 + R^{\Sigma} u^2 ) dv + 2 \int_{\partial \Sigma} H^{\partial \Sigma} u^2 ds}{\Big(\ds \int_{\Sigma} u^{\frac{2n}{n-2}} dv \Big)^{\frac{n-2}{n}}} + 2c \operatorname{vol}(\Sigma)^{\frac{2}{n}} = Q_h^{1,0}(u) + 2c \operatorname{vol}(\Sigma)^{\frac{2}{n}},
$$
for all \(u \in C^{\infty}(\Sigma)\) positive. This implies that
\begin{align*}
0 &\leq \inf_{u \in C^\infty(\Sigma),\, u > 0} Q_h^{1,0}(u) + 2c \operatorname{vol}(\Sigma)^{\frac{2}{n}} \\
&= Q_h^{1,0}(\Sigma, \partial \Sigma) + 2c \operatorname{vol}(\Sigma)^{\frac{2}{n}} \\
&\leq \sigma^{1,0}(\Sigma, \partial \Sigma) + 2c \operatorname{vol}(\Sigma)^{\frac{2}{n}},
\end{align*}
which gives
\begin{eqnarray*}
\operatorname{vol}(\Sigma) \geq \left( \frac{|\sigma^{1,0}(\Sigma, \partial \Sigma)|}{2 c} \right)^{\frac{n}{2}},
\end{eqnarray*}
as we intended to prove.

To prove item (2), we start from the inequality 
$$
0 \leq \int_{\Sigma}(2|\nabla u|^2 + (R^{\Sigma} - 2(\mu + J(N)) - |\chi^+|^2)u^2) dv 
- 2\int_{\partial\Sigma}(\operatorname{II}^{\partial M}(N, N) - \langle X, \nu \rangle)u^2 ds,
$$ 
which, as before, holds for all \( u \in C^\infty(\Sigma) \). In this case, since \( \mu + J(N) \geq 0 \) on \( \Sigma \) and
\begin{align*}
\operatorname{II}^{\partial M}(N, N) - \langle X, \nu \rangle &=H^{\partial M} - H^{\partial\Sigma} - (\iota_\varrho \pi)^\top(N)\\
&\geq H^{\partial M} - |(\iota_\varrho \pi)^\top| - H^{\partial\Sigma} \\
&\geq -\Bar{c} - H^{\partial\Sigma},
\end{align*}
we have
\[
0 \leq \int_{\Sigma}(2|\nabla u|^2 + R^{\Sigma}u^2) dv + 2\int_{\partial\Sigma} H^{\partial\Sigma}u^2 ds + 2\Bar{c}\int_{\partial\Sigma} u^2 ds 
\] 
for all \( u \in C^\infty(\Sigma) \). 

Once again, it follows from Hölder's inequality and the fact that \( a_n > 2 \) that 
\[
0 \leq \int_{\Sigma}(a_n|\nabla u|^2 + R^{\Sigma}u^2) dv + 2\int_{\partial\Sigma} H^{\partial\Sigma}u^2 ds 
+ 2\Bar{c}\operatorname{vol}(\partial\Sigma)^{\frac{1}{n-1}}\left(\int_{\partial\Sigma} u^{\frac{2(n-1)}{n-2}} ds\right)^{\frac{n-2}{n-1}}. 
\] 

Therefore, dividing the above inequality by 
\[
\left(\int_{\partial\Sigma} u^{\frac{2(n-1)}{n-2}} ds\right)^{\frac{n-2}{n-1}} 
\] 
and taking the infimum over all functions \( u \in C^\infty(\Sigma) \) with \( u > 0 \), we obtain: 
\[
0 \leq Q_h^{0,1}(\Sigma, \partial\Sigma) + 2\Bar{c}\operatorname{vol}(\partial\Sigma)^{\frac{1}{n-1}} 
\leq \sigma^{0,1}(\Sigma, \partial\Sigma) + 2\Bar{c}\operatorname{vol}(\partial\Sigma)^{\frac{1}{n-1}}, 
\] 
which establishes the result.
\end{proof}

The next proposition is an \textit{infinitesimal rigidity} result that we obtain by assuming the volume of \( \Sigma \) saturates inequality~\eqref{eq1:prop:aux1}.

\begin{prop}[Infinitesimal rigidity]\label{prop:aux2}
Under the hypotheses of Proposition~\ref{prop:aux1}, if equality in~(\ref{eq1:prop:aux1}) holds, then:
\begin{itemize}
\item $(\Sigma,h)$ is Einstein with scalar curvature $R^\Sigma = -2c$ and a totally geodesic boundary, where $h = g|_\Sigma$;
\item $\mu + J(N) = -c$ and $\chi^+ = 0$ on $\Sigma$;
\item $H^{\partial M} = |(\iota_\varrho\pi)^\top| = \langle X, \nu \rangle$ on $\partial\Sigma$;
\item $\lambda_1(\mathcal{L}_0) = 0$.
\end{itemize}
\end{prop}

\begin{proof}
From the solution of the Yamabe problem, we know that there exists a positive function $\phi \in C^\infty(\Sigma)$ that attains the Yamabe constant $Q_h^{1,0}(\Sigma, \partial\Sigma)$, \emph{i.e.}\ $$Q_h^{1,0}(\phi) = Q_h^{1,0}(\Sigma, \partial\Sigma).$$ 

On the other hand, it follows from the end of the proof of item~(1) of Proposition~\ref{prop:aux1}, along with equality in~\eqref{eq1:prop:aux1}, that
\begin{align}\label{eq1:prop:aux2}
0 \le Q_h^{1,0}(\Sigma, \partial\Sigma) + 2c \operatorname{vol}(\Sigma)^{\frac{2}{n}} \le \sigma^{1,0}(\Sigma, \partial\Sigma) + 2c \operatorname{vol}(\Sigma)^{\frac{2}{n}} = 0.
\end{align} 
Therefore,
\begin{align*}
Q_h^{1,0}(\phi) + 2c \operatorname{vol}(\Sigma)^{\frac{2}{n}} = Q_h^{1,0}(\Sigma, \partial\Sigma) + 2c \operatorname{vol}(\Sigma)^{\frac{2}{n}} = 0,
\end{align*}
which, by multiplying by $(\int_{\Sigma} \phi^{\frac{2n}{n-2}} dv )^{\frac{n-2}{n}}$, implies
\begin{align*}
\int_{\Sigma}(a_n|\nabla\phi|^2 + R^{\Sigma} \phi^2) dv + 2 \int_{\partial\Sigma} H^\Sigma \phi^2 ds + 2c \operatorname{vol}(\Sigma)^{\frac{2}{n}} \left( \int_{\Sigma} \phi^{\frac{2n}{n-2}} dv \right)^{\frac{n-2}{n}} = 0.
\end{align*}

Since $\Sigma$ is a stable MOTS, it follows from Lemma~\ref{lem:Mendes3} that
\begin{eqnarray*}
0 &\leq & 2 \lambda_1(\mathcal{L}_0) \int_\Sigma \phi^2 dv \\
&\leq & 2 \left( \int_{\Sigma} (|\nabla\phi|^2 + Q \phi^2) dv - \int_{\partial\Sigma} (\operatorname{II}^{\partial M}(N, N) - \langle X, \nu \rangle) \phi^2 ds \right) \\
&=& \int_{\Sigma} (2 |\nabla \phi|^2 + (R^\Sigma - 2(\mu + J(N)) - |\chi^+|^2) \phi^2) dv - 2 \int_{\partial\Sigma} (H^{\partial M} - H^{\partial\Sigma} - \langle X, \nu \rangle) \phi^2 ds \\
&\leq & \int_{\Sigma} (a_n |\nabla \phi|^2 + (R^\Sigma + 2c) \phi^2) dv + 2 \int_{\partial\Sigma} H^{\partial\Sigma} \phi^2 ds \\
&\leq & \int_{\Sigma} (a_n |\nabla \phi|^2 + R^\Sigma \phi^2) dv + 2 \int_{\partial\Sigma} H^{\partial\Sigma} \phi^2 ds + 2c \operatorname{vol}(\Sigma)^{\frac{2}{n}} \left( \int_{\Sigma} \phi^{\frac{2n}{n-2}} dv \right)^{\frac{n-2}{n}} \\
&=& 0,
\end{eqnarray*}
where, above, we have used that $a_n > 2$, $\mu + J(N) \geq -c$ on $\Sigma$, $\II^{\partial M}(N, N) = H^{\partial M} - H^\Sigma$, and 
$$
H^{\partial M} - \langle X, \nu \rangle = H^{\partial M} - (\iota_\varrho \pi)^\top(N) \geq H^{\partial M} - |(\iota_\varrho \pi)^\top| \geq 0\quad\mbox{on}\quad\partial\Sigma.
$$
Therefore, all the above inequalities are indeed equalities. Hence, we obtain:
\begin{itemize}
\item $\lambda_1(\mathcal{L}_0) = 0$;
\item $\mu + J(N) = -c$ and $\chi^+ = 0$ on $\Sigma$;
\item $H^{\partial M} = |(\iota_\varrho \pi)^\top| = \langle X, \nu \rangle$ on $\partial\Sigma$.
\end{itemize}
Additionally, $|\nabla \phi|^2 \equiv 0$, that is, $\phi$ is constant, since $a_n > 2$. 

Finally, from~\eqref{eq1:prop:aux2}, we conclude that $Q_h^{1,0}(\Sigma, \partial\Sigma) = \sigma^{1,0}(\Sigma, \partial\Sigma)$. Thus, by Proposition~\ref{prop:CruzSantos}, we obtain that the Yamabe metric $\Bar{h} = \phi^{\frac{4}{n-2}} h$ on $\Sigma$ is Einstein with a totally geodesic boundary; the same holds for $h$, since $\phi$ is constant, which concludes the proof of the proposition.
\end{proof}

For the next lemma, consider the operator
$$
L^*\psi = -\Delta\psi - 2\langle X, \nabla\psi \rangle + (Q - |X|^2 - \operatorname{div} X) \psi, \quad \psi \in C^\infty(\Sigma),
$$
called the \textit{formal adjoint operator} of $L$; this is because, by the divergence theorem,
$$
\int_\Sigma \phi L^*\psi \, dv + \int_{\partial\Sigma} \phi B^*\psi \, ds=\int_\Sigma \psi L\phi \, dv + \int_{\partial\Sigma} \psi B\phi \, ds
$$
for all $\phi, \psi \in C^\infty(\Sigma)$, where
$$
B^*\psi := \frac{\partial\psi}{\partial\nu} - ( \operatorname{II}^{\partial M}(N,N) - 2\langle X, \nu \rangle ) \psi.
$$

\begin{lem}\label{lem:aux1}
Under the assumptions of Proposition~\ref{prop:aux2}, we have that $\lambda = 0$ is a simple eigenvalue of $L$ on $\Sigma$ with Robin boundary condition $B\phi = 0$, and its associated eigenfunctions can be chosen positive. The same holds for the formal adjoint operator $L^*$ of $L$ on $\Sigma$ with Robin boundary condition $B^*\phi^* = 0$.
\end{lem}

The proof of the above lemma is entirely analogous to the proof of Lemma~3.5 in~\cite{Mendes2022} (with Proposition~\ref{prop:aux2} playing the role of Proposition~3.3), and therefore it will be omitted.

\begin{lem}\label{lem:aux2}
Under the hypotheses of Prop.~\ref{prop:aux2}, there exist a neighborhood $V \cong (-\delta, \delta) \times \Sigma$ of $\Sigma \cong \{0\} \times \Sigma$ in $M$ and a positive function $\varphi: V \to \mathbb{R}$ such that:
\begin{enumerate}
\item[$(1)$] $g|_V$ has the orthogonal decomposition,
$$
g|_V = \varphi^2 dt^2 + h_t,
$$
where $h_t$ is the induced metric on $\Sigma_t \cong \{t\} \times \Sigma$;
\item[$(2)$] Each $\Sigma_t$ is a free boundary hypersurface in $(M, g, K)$ with constant null mean\linebreak curvature $\theta^+(t)$ with respect to the outward unit normal $N_t = \varphi^{-1} \frac{\partial}{\partial t}$, where $N_0 = N$;
\item[$(3)$] $\frac{\partial \varphi}{\partial \nu_t} = \mathrm{II}^{\partial M}(N_t, N_t) \varphi$ on $\partial \Sigma_t$, where $\nu_t$ is the outward unit normal to $\partial \Sigma_t$ in $(\Sigma_t, h_t)$.
\end{enumerate}
\end{lem}

Once again, the proof of the above lemma is entirely analogous to the proof of Lemma~3.6 in~\cite{Mendes2022} (with Lemma~\ref{lem:aux1} playing the role of Lemma~3.5), and therefore it will also be omitted.

We can now proceed to the proof of Theorem~\ref{theo:B}.

\begin{proof}[Proof of Theorem~\ref{theo:B}]
First, since $\mu+J(N)\ge\mu-|J|\ge-c$ on $M_+$ and, in particular, on $\Sigma$, it follows from item (1) of Proposition~\ref{prop:aux1} that
\begin{align*}
\operatorname{vol}(\Sigma) \geq \left(\frac{|\sigma^{1,0}(\Sigma, \partial\Sigma)|}{2c}\right)^{\frac{n}{2}}.
\end{align*}
Furthermore, if equality holds, Proposition~\ref{prop:aux2} furnishes that $(\Sigma,h)$ is Einstein with scalar curvature $R^\Sigma=-2c$ and a totally geodesic boundary. Moreover, we can apply Lemma~\ref{lem:aux2} to guarantee the existence of a foliation $(\Sigma_t)_{0\le t<\delta}$ of an outer neighborhood $V$ of $\Sigma = \Sigma_0$ in~$M$ such that each leaf $\Sigma_t$ is a free boundary hypersurface with constant null mean curvature $\theta^+=\theta^+(t)$ in $(M,g,K)$. Furthermore, the metric $g$ can be written as
$$
g = \varphi^2 dt^2 + h_t\quad\mbox{on}\quad V \cong [0,\delta) \times \Sigma,
$$
where $h_t = g|_{\Sigma_t}$ is the induced metric on $\Sigma_t$.

On the other hand, from~\eqref{eq:variacaotheta}, it follows that
\begin{align}\label{eq1:theo:main2}
\frac{d\theta^+}{dt} = -\Delta\varphi + 2\langle X, \nabla\varphi \rangle + \Big(Q - |X|^2 + \div X - \frac{1}{2}(\theta^+)^2 + \theta^+ \tau \Big) \varphi\quad\mbox{on}\quad\Sigma_t,
\end{align}
for each $t\in[0,\delta)$. Thus, taking $Y = X - \varphi^{-1} \nabla \varphi$ and observing that
\begin{align*}
\div Y = \div X - \frac{\Delta\varphi}{\varphi} + \frac{|\nabla\varphi|^2}{\varphi^2}\quad\mbox{and}\quad
|Y|^2 = |X|^2 - 2\frac{\langle X, \nabla\varphi \rangle}{\varphi} + \frac{|\nabla\varphi|^2}{\varphi^2},
\end{align*}
we get, from~\eqref{eq1:theo:main2}, that
\begin{align*}
\frac{1}{\varphi}\frac{d\theta^+}{dt} = \div Y - |Y|^2 + Q - \frac{1}{2}(\theta^+)^2 + \theta^+ \tau \le \div Y - |Y|^2 + Q + \theta^+ \tau\quad\mbox{on}\quad\Sigma_t,
\end{align*}
for each $t \in [0,\delta)$. Therefore, given $u \in C^\infty(\Sigma_t)$, we have:
\begin{eqnarray}
\frac{u^2}{\varphi}\frac{d\theta^+}{dt} - u^2 \theta^+ \tau &\le& u^2 \div Y - u^2 |Y|^2 + Q u^2 \nonumber \\
&=& \div(u^2 Y) - 2u \langle \nabla u, Y \rangle - u^2 |Y|^2 + \Big(\frac{1}{2} R^{\Sigma_t} - (\mu + J(N_t)) - \frac{1}{2} |\chi_t^+|^2\Big) u^2 \nonumber \\
&\le& \div(u^2 Y) + 2|u||\nabla u||Y| - u^2 |Y|^2 + \Big(\frac{1}{2} R^{\Sigma_t} - (\mu - |J|)\Big) u^2 \nonumber \\
&\le& \div(u^2 Y) + |\nabla u|^2 + \left(\frac{1}{2} R^{\Sigma_t} + c\right) u^2, \label{eq2:theo:main2}
\end{eqnarray}
where, above, we have used the Cauchy-Schwarz inequality, along with the fact that
\begin{align*}
\mu + J(N_t) \ge \mu - |J| \ge -c\quad\mbox{on}\quad M_+.
\end{align*}
Thus, integrating~\eqref{eq2:theo:main2} over $\Sigma_t$ and noting that $\frac{d\theta^+}{dt}$ and $\theta^+$ are constants on $\Sigma_t$, we obtain
\begin{eqnarray}
2\left(\frac{d\theta^+}{dt} \int_{\Sigma_t} \frac{u^2}{\varphi} dv_t - \theta^+ \int_{\Sigma_t} \tau u^2 dv_t\right) &\le& \int_{\Sigma_t} (2|\nabla u|^2 + R^{\Sigma_t} u^2) dv_t + 2c \int_{\Sigma_t} u^2 dv_t \nonumber \\
&& + 2 \int_{\partial \Sigma_t} u^2 \langle Y, \nu_t \rangle ds_t, \label{eq3:theo:main2}
\end{eqnarray}
where, above, we have used the divergence theorem.

Now, replacing \( Y \) by \( X - \varphi^{-1} \nabla \varphi \) and using that
$$
\dfrac{\partial \varphi}{\partial \nu_t} = \operatorname{II}^{\partial M}(N_t, N_t) \varphi = (H^{\partial M} - H^{\partial \Sigma_t}) \varphi
$$
(see Lemma~\ref{lem:aux2}), we obtain:
\begin{align*}
\int_{\partial \Sigma_t} u^2 \langle Y, \nu_t \rangle ds_t &= \int_{\partial \Sigma_t} u^2 \langle X - \varphi^{-1} \nabla \varphi, \nu_t \rangle ds_t \\
&= \int_{\partial \Sigma_t} u^2 ( \langle X, \nu_t \rangle + H^{\partial \Sigma_t} - H^{\partial M} ) ds_t.
\end{align*}
On the other hand, using the BDEC as done in the proof of Proposition~\ref{prop:aux1}, we get that \( \langle X, \nu_t \rangle - H^{\partial M} \leq 0 \) on $\partial\Sigma_t$. Therefore,
\begin{equation}\label{eq4:theo:main2}
\int_{\partial \Sigma_t} u^2 \langle Y, \nu_t \rangle ds_t \leq \int_{\partial \Sigma_t} u^2 H^{\partial \Sigma_t} ds_t.
\end{equation}
Thus, substituting (\ref{eq4:theo:main2}) into (\ref{eq3:theo:main2}) and using the fact that \( a_n = \frac{4(n-1)}{n-2} > 2 \) for all \( n \geq 3 \), along with Hölder's inequality, we obtain
\begin{eqnarray}
2\left(\frac{d\theta^+}{dt} \int_{\Sigma_t} \frac{u^2}{\varphi} dv_t - \theta^+ \int_{\Sigma_t} \tau u^2 dv_t\right) &\le& \int_{\Sigma_t}( a_n |\nabla u|^2 + R^{\Sigma_t} u^2 ) dv_t + 2 \int_{\partial \Sigma_t} u^2 H^{\partial \Sigma_t} ds_t \nonumber \\
&& + 2c \operatorname{vol}(\Sigma_t)^{\frac{2}{n}} \left( \int_{\Sigma_t} u^{\frac{2n}{n-2}} dv_t \right)^{\frac{n-2}{n}} \label{eq5:theo:main2}
\end{eqnarray}
for all \( u \in C^\infty(\Sigma_t) \) and all \( t \in [0, \delta) \).

From the solution to the Yamabe problem, together with Theorem~\ref{theo:Escobar}, we know that, for each \( t \in [0, \delta) \), there exists a unique Yamabe metric \( \Bar{h}_t \) in the conformal class of \( h_t \) with constant scalar curvature \(-2c \) and minimal boundary. Let \( u_t \in C^\infty(\Sigma_t) \), \( u_t > 0 \), be such that \( \Bar{h}_t = u_t^{4/(n-2)} h_t \). Substituting \( u = u_t \) into equation~\eqref{eq5:theo:main2}, we obtain
\begin{eqnarray}
\frac{2\Big(\ds\frac{d\theta^+}{dt} \int_{\Sigma_t} \frac{u_t^2}{\varphi} dv_t - \theta^+ \int_{\Sigma_t} \tau u_t^2 dv_t \Big)}{\Big(\ds\int_{\Sigma_t} u_t^{\frac{2n}{n-2}} dv_t \Big)^{\frac{n-2}{n}}} &\le & \frac{\ds\int_{\Sigma_t}( a_n |\nabla u_t|^2 + R^{\Sigma_t} u_t^2) dv_t + 2 \int_{\partial \Sigma_t} u_t^2 H^{\partial \Sigma_t} ds_t}{\Big(\ds\int_{\Sigma_t} u_t^{\frac{2n}{n-2}} dv_t \Big)^{\frac{n-2}{n}}} \nonumber\\
&& + 2c \operatorname{vol}(\Sigma_t)^{\frac{2}{n}} \nonumber\\
&=& Q_{h_t}^{1,0}(u_t) + 2c \operatorname{vol}(\Sigma_t)^{\frac{2}{n}} \nonumber\\
&=& Q_{h_t}^{1,0}(\Sigma_t, \partial \Sigma_t) + 2c \operatorname{vol}(\Sigma_t)^{\frac{2}{n}} \nonumber\\
&\le& \sigma^{1,0}(\Sigma, \partial \Sigma) + 2c \operatorname{vol}(\Sigma_t)^{\frac{2}{n}} \nonumber\\
&=& 2c \operatorname{vol}(\Sigma_t)^{\frac{2}{n}} - 2c \operatorname{vol}(\Sigma)^{\frac{2}{n}} \nonumber\\
&=& 2c \int_0^t \frac{d}{dr}( \operatorname{vol}(\Sigma_r)^{\frac{2}{n}} ) dr \nonumber\\
&=& \frac{4c}{n} \int_0^t \operatorname{vol}(\Sigma_r)^{\frac{2-n}{n}} \frac{d}{dr} \operatorname{vol}(\Sigma_r) dr\label{eq.4.11}
\end{eqnarray}
for each \( t \in [0, \delta) \). Above, we have used that \( \Sigma_t \) is diffeomorphic to \( \Sigma \) and \( \sigma^{1,0}(\Sigma, \partial \Sigma) \) is a topological invariant of \( \Sigma \). We have also applied the fundamental theorem of calculus and the assumption that
\begin{align*}
\operatorname{vol}(\Sigma) = \left(\frac{|\sigma^{1,0}(\Sigma, \partial\Sigma)|}{2c}\right)^{\frac{n}{2}}.
\end{align*}

On the other hand, using the assumption that \( K \) is \( n \)-convex on \( M_+ \), we have \( \text{tr}_{\Sigma_t}K \ge 0 \) for each \( t \in [0,\delta) \). Thus, \( \theta^+(t) = \text{tr}_{\Sigma_t}K + H^{\Sigma_t} \ge H^{\Sigma_t} \). Therefore, from the first variation of volume formula, along with~\eqref{eq.4.11}, we get
\begin{align*}
\frac{\ds\frac{d\theta^+}{dt}\int_{\Sigma_t}\frac{u_t^2}{\varphi}dv_t - \theta^+ \int_{\Sigma_t} \tau u_t^2 dv_t}{\Big(\ds\int_{\Sigma_t} u_t^{\frac{2n}{n-2}} dv_t \Big)^{\frac{n-2}{n}}} &\le \frac{2c}{n} \int_0^t \operatorname{vol}(\Sigma_r)^{\frac{2-n}{n}} \Big( \int_{\Sigma_r} H^{\Sigma_r} \varphi dv_r \Big) dr\\
&\le \frac{2c}{n} \int_0^t \theta^+(r) \Big( \operatorname{vol}(\Sigma_r)^{\frac{2-n}{n}} \int_{\Sigma_r} \varphi dv_r \Big) dr.
\end{align*}

Thus, defining
\[
\begin{gathered}
f(t) = \theta^+(t), \quad \eta(t) = \Big( \int_{\Sigma_t} \frac{u_t^2}{\varphi} dv_t \Big) \Big( \int_{\Sigma_t} u_t^{\frac{2n}{n-2}} dv_t \Big)^{-\frac{n-2}{n}}, \\
\rho(t) = \Big( \int_{\Sigma_t} \tau u_t^2 dv_t \Big) \Big( \int_{\Sigma_t} u_t^{\frac{2n}{n-2}} dv_t \Big)^{-\frac{n-2}{n}},
\end{gathered}
\]
and
\[
\xi(t) = \frac{2c}{n} \operatorname{vol}(\Sigma_t)^{\frac{2-n}{n}} \int_{\Sigma_t} \varphi dv_t,
\]
we can apply Lemma~\ref{lem:Mendes4} to obtain that \( \theta^+(t) \le 0 \) for each \( t \in [0,\delta) \). However, since \( \Sigma \) is weakly outermost and \( \theta^+(t) \) is constant on \( \Sigma_t \), it follows that \( \theta^+(t) = 0 \) for all \( t \in [0,\delta) \). Therefore, all the above inequalities are, in fact, equalities. 

In particular, for each $t\in[0,\delta)$, we have:
\begin{itemize}
\item \( \operatorname{div} Y - |Y|^2 + Q = 0 \) on $\Sigma_t$;
\item \( \chi_t^+ \equiv 0 \);
\item \( \mu + J(N_t) = \mu - |J| = -c \) on $\Sigma_t$;
\item \(H^{\Sigma_t}=\theta^+(t)=0\), \emph{i.e.}\ $\Sigma_t$ is a minimal MOTS (in particular, $\operatorname{tr}_{\Sigma_t}K=0$);
\item \( \langle X, \nu_t \rangle - H^{\partial M} = 0 \) on $\partial\Sigma_t$;
\item \(\operatorname{vol}(\Sigma_t)=\operatorname{vol}(\Sigma)\).
\end{itemize}

Additionally, since \( \Sigma_t \) is a MOTS for each \( t \in [0,\delta) \), we can use \eqref{eq:variacaotheta} again to obtain
\[
L \varphi = \frac{d\theta^+}{dt} = 0 \quad \text{on} \quad \Sigma_t.
\]
By Lemma~\ref{lem:aux2}, we get
\[
\frac{\partial \varphi}{\partial \nu_t} - \operatorname{II}^{\partial M}(N_t, N_t) \varphi = 0 \quad \text{on} \quad \partial \Sigma_t.
\]
Thus, we conclude that \( \Sigma_t \) is stable for each \( t \in [0,\delta) \). Therefore, we can apply Proposition~\ref{prop:aux2} to \( \Sigma_t \) and conclude that \( H^{\partial \Sigma_t} = 0 \) on $\partial\Sigma_t$ and \( Q = 0 \) on $\Sigma_t$. (From Proposition~\ref{prop:aux2}, we also have that $H^{\partial M} = |(\iota_\varrho\pi)^\top|$ on $\partial\Sigma_t$ for each $t\in[0,\delta)$, that is, the BDEC is saturated along $V\cap\partial M$.)

From \( \operatorname{div} Y - |Y|^2 + Q = 0 \) and \( Q = 0 \), we get \( \operatorname{div} Y - |Y|^2 = 0 \), and so
\begin{align*}
\int_{\Sigma_t} |Y|^2 dv_t &= \int_{\Sigma_t} \operatorname{div} Y dv_t \\
&= \int_{\partial \Sigma_t} \langle Y, \nu_t \rangle ds_t \\
&= \int_{\partial \Sigma_t} ( \langle X, \nu_t \rangle + H^{\partial \Sigma_t} - H^{\partial M} ) ds_t \\
&= \int_{\partial \Sigma_t}( \langle X, \nu_t \rangle - H^{\partial M}) ds_t \\
&= 0.
\end{align*}
Hence \( Y = X - \varphi^{-1} \nabla \varphi = 0 \).

Since $\Sigma_t$ is a minimal MOTS, we also have that \( \theta^-(t) = \operatorname{tr}_{\Sigma_t} K - H^{\Sigma_t}=0 \) for each $t\in[0,\delta)$. Then, by replacing \( \theta^{+} \) and \( \varphi \) by \( \theta^{-} \) and \( \varphi^{-} = -\varphi \) into (\ref{eq:variacaotheta}), respectively, we obtain:
\begin{align}
0 = \frac{d\theta^{-}}{dt} = -\Delta \varphi^{-} + 2 \langle X^{-}, \nabla \varphi^{-} \rangle + ( Q^{-} - |X^{-}|^2 + \operatorname{div} X^{-}) \varphi^{-},\label{eq3.2}
\end{align}
where
\begin{align}
Q^{-} &= \frac{1}{2} R^{\Sigma_t} - (\mu + J(-N_t)) - \frac{1}{2} |\chi_t^{-}|^2\nonumber\\
&= -c - (\mu + |J|) - \frac{1}{2} |\chi_t^{-}|^2\nonumber\\
&= -2|J| - \frac{1}{2} |\chi_t^{-}|^2,\label{eq3.3}
\end{align}
and
\begin{align}
X^{-} = -X = -\frac{1}{\varphi} \nabla \varphi.\label{eq3.4}
\end{align}

Substituting equations (\ref{eq3.3}) and (\ref{eq3.4}) into (\ref{eq3.2}), we get:
\[
\Delta \varphi + \frac{|\nabla \varphi|^2}{\varphi} + \Big( |J| + \frac{1}{4} |\chi_t^{-}|^2 \Big) \varphi = 0.
\]
Since
\[
\frac{\partial \varphi}{\partial \nu_t} = \operatorname{II}^{\partial M}(N_t, N_t) \varphi = (H^{\partial M} - H^{\partial \Sigma_t}) \varphi = H^{\partial M} \varphi \geq 0,
\]
integrating over \( \Sigma_t \), and using the Divergence Theorem, we have:
\[
0 \leq \int_{\partial \Sigma_t} \frac{\partial \varphi}{\partial \nu_t} ds_t = \int_{\Sigma_t} \Delta \varphi dv_t = - \int_{\Sigma_t} \left( \frac{|\nabla \varphi|^2}{\varphi} + \Big( |J| + \frac{1}{4} |\chi_t^{-}|^2 \Big) \varphi \right) dv_t \leq 0,
\]
so
\[
|\nabla \varphi| = |\chi_t^{-}| = |J| = 0 \quad \text{on}\quad V.
\]

Thus, \( \varphi \) is constant on $\Sigma_t$ and, since \( X = -\varphi^{-1} \nabla \varphi = 0 \), we have \( K(N_t, \cdot) |_{\Sigma_t} = 0 \). Furthermore, since \( \chi_t^+ = K|_{\Sigma_t} + A_t\equiv0 \) and \( \chi_t^- = K|_{\Sigma_t} - A_t\equiv0 \), we have \( K|_{\Sigma_t} \equiv 0 \).

We conclude that \(\Sigma_t\) is totally geodesic in \( (M, g) \) for each \( t \in [0, \delta) \). Writing \( g = \varphi^2 dt^2 + h_t \) on \( V \cong [0, \delta) \times \Sigma \), and noting that \( \varphi = \varphi_t \) depends only on \( t \in [0, \delta) \) and that \( h_t \) does not depend on \( t \), since \( \Sigma_t \) is totally geodesic for each $t$, we can see that, by a change of variables, \( g \) has the structure of a product metric \( dt^2 + h \) on \( V \), where \( (\Sigma, h) \) is Einstein with scalar curvature \( R^\Sigma = -2c \) and a totally geodesic boundary.

Finally, using that \( 0 = J = \operatorname{div}(K - (\operatorname{tr} K) g) = \operatorname{div} K - d(\operatorname{tr} K) \), \( K |_{\Sigma_t}\equiv 0\), $A_t\equiv0$, and $K(N_t, \cdot)|_{\Sigma_t} \equiv 0$, we conclude that \( K = a\, dt^2 \) on \( V \), where \( a \) depends only on \( t \in [0, \delta) \).
\end{proof}

\begin{example}
Let \((\Sigma^n, h_{\text{hyp}})\) be an \(n\)-dimensional closed hyperbolic Riemannian manifold, where \(n \geq 2\). The \((n+2)\)-dimensional anti-Nariai spacetime \((\bar{M}, \bar{g})\) is given by the manifold  
\[
\bar{M} = \mathbb{R} \times (0, \infty) \times \Sigma^n
\]
equipped with the Lorentzian metric  
\begin{align*}  
\bar{g} = \frac{n}{2(-\Lambda)} \left( -\sinh^2 \chi \, dt^2 + d\chi^2 + (n-1) h_{\text{hyp}} \right),  
\end{align*}  
which satisfies the Einstein vacuum equation  
\begin{align*}
\operatorname{Ric}_{\bar{M}}=\frac{2\Lambda}{n}\bar{g},
\end{align*}
where \(\Lambda\) is a negative cosmological constant and \(\operatorname{Ric}_{\bar{M}}\) denotes the Ricci tensor of \((\bar{M},\bar{g})\) (See~\cite{CardosoDiasLemos} for a detailed description of the anti-Nariai spacetimes.)  

It is straightforward to verify that the second fundamental form \(K\) of the slice \(M = \{t = t_0\}\) vanishes. Furthermore, the local energy density \(\mu\) and the local current density \(J\) of \((M, g, K)\) (where \(g\) is the induced metric on \(M\)) are given by \(\mu = \Lambda\) and \(J = 0\).  

Now, assume that \((\Sigma^n, h_{\text{hyp}})\) realizes the Yamabe invariant (this always holds for \(n = 2\) and \(n = 3\), see \emph{e.g.}\ \cite{Anderson2006}):  
\begin{align*}
\sigma(\Sigma) = \frac{\displaystyle\int_\Sigma R^\Sigma dv}{\operatorname{vol}(\Sigma)^{\frac{n-2}{n}}} = -n(n-1) \operatorname{vol}(\Sigma)^{\frac{2}{n}} = 2\Lambda \operatorname{vol}(\Sigma, h)^{\frac{2}{n}},
\end{align*}
where \(h = \frac{n(n-1)}{2(-\Lambda)} h_{\text{hyp}}\) is the induced metric on \(\Sigma\). Here, $R^\Sigma=-n(n-1)$, $dv$, and $\operatorname{vol}(\Sigma)$ are the scalar curvature, the volume element, and the volume of $\Sigma$ with respect to $h_{\text{hyp}}$. Therefore,  
\begin{align*}
\operatorname{vol}(\Sigma, h) = \left(\frac{|\sigma(\Sigma)|}{2(-\Lambda)}\right)^{\frac{n}{2}}.
\end{align*}

This provides an example of an initial data set that satisfies all the hypotheses and saturates the inequality of Theorem~\ref{theo:Mendes1} (for \(n=2\)) and Theorem~\ref{theo:Mendes2} (for \(n \geq 3\)).  

In a similar manner, we can construct an initial data set satisfying all the hypotheses of Theorem~\ref{theo:B}, which also saturates the respective inequality, by considering a compact hyperbolic manifold \((\Sigma^n, h_{\text{hyp}})\) with a totally geodesic boundary. 
\end{example}

\bibliographystyle{amsplain}
\bibliography{bibliography2.bib}

\providecommand{\bysame}{\leavevmode\hbox to3em{\hrulefill}\thinspace}
\providecommand{\MR}{\relax\ifhmode\unskip\space\fi MR }
\providecommand{\MRhref}[2]{%
  \href{http://www.ams.org/mathscinet-getitem?mr=#1}{#2}
}
\providecommand{\href}[2]{#2}
\begin{thebibliography}{10}

\bibitem{AlaeeLesourdYau2021}
A.~Alaee, M.~Lesourd, and S.-T. Yau, \emph{Stable surfaces and free boundary
  marginally outer trapped surfaces}, Calc. Var. Partial Differ. Equ.
  \textbf{60} (2021), no.~5, 27 (English), Id/No 186.

\bibitem{AlmarazLimaMari2021}
S.~Almaraz, L.~L. de~Lima, and L.~Mari, \emph{Spacetime positive mass theorems
  for initial data sets with non-compact boundary}, Int. Math. Res. Not.
  \textbf{2021} (2021), no.~4, 2783--2841 (English).

\bibitem{Ambrozio2015}
L.~C. Ambrozio, \emph{Rigidity of area-minimizing free boundary surfaces in
  mean convex three-manifolds}, J. Geom. Anal. \textbf{25} (2015), no.~2,
  1001--1017 (English).

\bibitem{Anderson2006}
Michael~T. Anderson, \emph{Canonical metrics on 3-manifolds and 4-manifolds},
  Asian J. Math. \textbf{10} (2006), no.~1, 127--164 (English).

\bibitem{AnderssonMarsSimon2008}
L.~Andersson, M.~Mars, and W.~Simon, \emph{Stability of marginally outer
  trapped surfaces and existence of marginally outer trapped tubes}, Adv.
  Theor. Math. Phys. \textbf{12} (2008), no.~4, 853--888 (English).

\bibitem{AnderssonMetzger2009}
L.~Andersson and J.~Metzger, \emph{The area of horizons and the trapped
  region}, Commun. Math. Phys. \textbf{290} (2009), no.~3, 941--972 (English).

\bibitem{Aubin1976}
T.~Aubin, \emph{Equations diff{\'e}rentielles non lin{\'e}aires et probl{\`e}me
  de {Yamabe} concernant la courbure scalaire}, J. Math. Pures Appl. (9)
  \textbf{55} (1976), 269--296 (French).

\bibitem{BarrosCruz2020}
A.~Barros and C.~Cruz, \emph{Free boundary hypersurfaces with non-positive
  {Yamabe} invariant in mean convex manifolds}, J. Geom. Anal. \textbf{30}
  (2020), no.~4, 3542--3562 (English).

\bibitem{BarrosCruzBatistaSousa2015}
A.~Barros, C.~Cruz, R.~Batista, and P.~Sousa, \emph{Rigidity in dimension four
  of area-minimising {Einstein} manifolds}, Math. Proc. Camb. Philos. Soc.
  \textbf{158} (2015), no.~2, 355--363 (English).

\bibitem{BrayBrendleNeves2010}
H.~Bray, S.~Brendle, and A.~Neves, \emph{Rigidity of area-minimizing
  two-spheres in three-manifolds}, Commun. Anal. Geom. \textbf{18} (2010),
  no.~4, 821--830 (English).

\bibitem{BrendleChen2014}
S.~Brendle and S.~S. Chen, \emph{An existence theorem for the {Yamabe} problem
  on manifolds with boundary}, J. Eur. Math. Soc. (JEMS) \textbf{16} (2014),
  no.~5, 991--1016 (English).

\bibitem{Cai2002}
M.~Cai, \emph{Volume minimizing hypersurfaces in manifolds of nonnegative
  scalar curvature}, Minimal surfaces, geometric analysis and symplectic
  geometry. Based on the lectures of the workshop and conference, Johns Hopkins
  University, Baltimore, MD, USA, March 16--21, 1999, Tokyo: Mathematical
  Society of Japan, 2002, pp.~1--7 (English).

\bibitem{CaiGalloway2000}
M.~Cai and G.~J. Galloway, \emph{Rigidity of area minimizing tori in
  3-manifolds of nonnegative scalar curvature}, Commun. Anal. Geom. \textbf{8}
  (2000), no.~3, 565--573 (English).

\bibitem{CardosoDiasLemos}
Vitor Cardoso, \'Oscar J.~C. Dias, and Jos\'e P.~S. Lemos, \emph{Nariai,
  bertotti-robinson, and anti-nariai solutions in higher dimensions}, Phys.
  Rev. D \textbf{70} (2004), 024002.

\bibitem{Chen2010}
S.~S. Chen, \emph{Conformal deformation to scalar flat metrics with constant
  mean curvature on the boundary in higher dimensions}, 2010.

\bibitem{CruzSantos2023}
T.~Cruz and A.~S. Santos, \emph{Critical metrics and curvature of metrics with
  unit volume or unit area of the boundary}, J. Geom. Anal. \textbf{33} (2023),
  no.~1, 34 (English), Id/No 22.

\bibitem{Almaraz2010}
S.~de~Moura~Almaraz, \emph{An existence theorem of conformal scalar-flat
  metrics on manifolds with boundary}, Pac. J. Math. \textbf{248} (2010),
  no.~1, 1--22 (English).

\bibitem{Deng2021}
H.~Deng, \emph{A {B}ray-{B}rendle-{N}eves type inequality for a {Riemannian}
  manifold}, Acta Math. Sci., Ser. B, Engl. Ed. \textbf{41} (2021), no.~2,
  487--492 (English).

\bibitem{EckerHuisken1991}
K.~Ecker and G.~Huisken, \emph{Interior estimates for hypersurfaces moving by
  mean curvature}, Invent. Math. \textbf{105} (1991), no.~3, 547--569
  (English).

\bibitem{Escobar1992v2}
J.~F. Escobar, \emph{Conformal deformation of a {Riemannian} metric to a scalar
  flat metric with constant mean curvature on the boundary}, Ann. Math. (2)
  \textbf{136} (1992), no.~1, 1--50 (English).

\bibitem{Escobar1992}
\bysame, \emph{The {Yamabe} problem on manifolds with boundary}, J. Differ.
  Geom. \textbf{35} (1992), no.~1, 21--84 (English).

\bibitem{Escobar2003}
\bysame, \emph{Uniqueness and non-uniqueness of metrics with prescribed scalar
  and mean curvature on compact manifolds with boundary.}, J. Funct. Anal.
  \textbf{202} (2003), no.~2, 424--442 (English).

\bibitem{Fischer-ColbrieSchoen1980}
D.~Fischer-Colbrie and R.~Schoen, \emph{The structure of complete stable
  minimal surfaces in 3-manifolds of non- negative scalar curvature}, Commun.
  Pure Appl. Math. \textbf{33} (1980), 199--211 (English).

\bibitem{Galloway2011}
G.~J. Galloway, \emph{Stability and rigidity of extremal surfaces in
  {Riemannian} geometry and general relativity}, Surveys in geometric analysis
  and relativity. Dedicated to Richard Schoen in honor of his 60th birthday,
  Somerville, MA: International Press; Beijing: Higher Education Press, 2011,
  pp.~221--239 (English).

\bibitem{GallowayJang2020}
G.~J. Galloway and H.~C. Jang, \emph{Some scalar curvature warped product
  splitting theorems}, Proc. Am. Math. Soc. \textbf{148} (2020), no.~6,
  2617--2629 (English).

\bibitem{GallowayMendes2024}
G.~J. Galloway and A.~Mendes, \emph{Aspects of the geometry and topology of
  expanding horizons}, Proc. Am. Math. Soc., to appear.

\bibitem{GallowaySchoen2006}
G.~J. Galloway and R.~Schoen, \emph{A generalization of {Hawking}'s black hole
  topology theorem to higher dimensions}, Commun. Math. Phys. \textbf{266}
  (2006), no.~2, 571--576 (English).

\bibitem{Marques2005}
F.~C. Marques, \emph{Existence results for the {Yamabe} problem on manifolds
  with boundary}, Indiana Univ. Math. J. \textbf{54} (2005), no.~6, 1599--1620
  (English).

\bibitem{Marques2007}
\bysame, \emph{Conformal deformations to scalar-flat metrics with constant mean
  curvature on the boundary}, Commun. Anal. Geom. \textbf{15} (2007), no.~2,
  381--405 (English).

\bibitem{Mendes2019TAMS}
A.~Mendes, \emph{Rigidity of marginally outer trapped (hyper)surfaces with
  negative {{\(\sigma \)}}-constant}, Trans. Am. Math. Soc. \textbf{372}
  (2019), no.~8, 5851--5868 (English).

\bibitem{Mendes2019}
\bysame, \emph{Rigidity of volume-minimising hypersurfaces in {Riemannian}
  5-manifolds}, Math. Proc. Camb. Philos. Soc. \textbf{167} (2019), no.~2,
  345--353 (English).

\bibitem{Mendes2022}
\bysame, \emph{Rigidity of free boundary {MOTS}}, Nonlinear Anal., Theory
  Methods Appl., Ser. A, Theory Methods \textbf{220} (2022), 15 (English),
  Id/No 112841.

\bibitem{MicallefMoraru2015}
M.~Micallef and V.~Moraru, \emph{Splitting of 3-manifolds and rigidity of
  area-minimising surfaces}, Proc. Am. Math. Soc. \textbf{143} (2015), no.~7,
  2865--2872 (English).

\bibitem{Moraru2016}
V.~Moraru, \emph{On area comparison and rigidity involving the scalar
  curvature}, J. Geom. Anal. \textbf{26} (2016), no.~1, 294--312 (English).

\bibitem{Nunes2013}
I.~Nunes, \emph{Rigidity of area-minimizing hyperbolic surfaces in
  three-manifolds}, J. Geom. Anal. \textbf{23} (2013), no.~3, 1290--1302
  (English).

\bibitem{PessoaVerasVieira}
L.~F. Pessoa, E.~Véras, and B.~Vieira, \emph{Area estimates for capillary cmc
  hypersurfaces with non-positive {Y}amabe invariant}, to appear.

\bibitem{Schoen1984}
R.~Schoen, \emph{Conformal deformation of a {Riemannian} metric to constant
  scalar curvature}, J. Differ. Geom. \textbf{20} (1984), 479--495 (English).

\bibitem{ShoenYau1979}
R.~Schoen and S.-T. Yau, \emph{Existence of incompressible minimal surfaces and
  the topology of three dimensional manifolds with non-negative scalar
  curvature}, Ann. Math. (2) \textbf{110} (1979), 127--142 (English).

\bibitem{ShoenYau1987}
\bysame, \emph{The structure of manifolds with positive scalar curvature},
  Directions in partial differential equations, {Proc}. {Symp}.,
  {Madison}/{Wis}. 1985, {Publ}. {Math}. {Res}. {Cent}. {Univ}. {Wis}.
  {Madison} 54, 235-242 (1987)., 1987.

\bibitem{Stahl1996}
A.~Stahl, \emph{Regularity estimates for solutions to the mean curvature flow
  with a {Neumann} boundary condition}, Calc. Var. Partial Differ. Equ.
  \textbf{4} (1996), no.~4, 385--407 (English).

\bibitem{Trudinger1968}
N.~S. Trudinger, \emph{Remarks concerning the conformal deformation of
  {Riemannian} structures on compact manifolds}, Ann. Sc. Norm. Super. Pisa,
  Sci. Fis. Mat., III. Ser. \textbf{22} (1968), 265--274 (English).

\bibitem{Yamabe1960}
H.~Yamabe, \emph{On a deformation of {Riemannian} structures on compact
  manifolds}, Osaka Math. J. \textbf{12} (1960), 21--37 (English).

\end{thebibliography}

\end{document}